\documentclass[hidelinks,onefignum,onetabnum]{siamart251216}

\usepackage{graphicx}%
\usepackage{multirow}%
\usepackage{amsmath,amssymb,amsfonts}%

\usepackage{mathrsfs}%
\usepackage{booktabs}%
\usepackage{float}
\usepackage{multirow}
\usepackage{graphicx,epsfig}
\usepackage{color}

\usepackage{epsfig,subfigure,epstopdf}
\usepackage{graphicx,euscript}

\usepackage{multirow,array}
\usepackage[draft,commentmarkup=margin]{changes}
\usepackage{cleveref}

\def\cM{\mathcal{M}}
\def\cN{\mathcal{N}}
\def\vu{\vec{u}}
\def\vnu{\vec{\nu}}
\def\hnu{\hat{\nu}}
\def\vx{\vec{x}}
\def\gM{\nabla_\cM}

\def\gGamma{\nabla_{\Gamma(t)}}

\def\gMh{\nabla_{\cM_h}}

\def\vX{\vec{X}}

\def\hX{\hat{X}}

\def\vxi{\vec{\xi}}

\def\Id{{\rm Id}}

\def\cI{\mathcal{I}}

\def\half{{1/2}}

\def\vn{\vec{n}}
\def\vw{\vec{w}}

\def\hnu{\hat{{\nu}}}

\def\mO{\mathcal{O}}
\def\mbfu{\mathbf{u}}
\def\mbfz{\mathbf{z}}

\def\hGam{\hat{\Gamma}_h^{m+\half}}

\def\mJ{\mathcal{J}}
\def\mS{\mathcal{S}}

\def\tn{\bar{n}}
\def\vz{\vec{z}}
\def\tr{{\rm{tr}}}
\def\d{{\rm{d}}}
\def\tT{\tilde{T}}
\def\tp{\tilde{p}}

\graphicspath{{figs/SDF/cuboid/}{figs/SDF/H/}}

\newsiamremark{remark}{Remark}

\theoremstyle{plain}
\newtheorem{example}{Example}[section]%

\allowdisplaybreaks[4] 
\headers{A Second-order Structure-preserving PFEM}{B.~Duan and Z.~Yang}
\title{A Second-order Structure-preserving Parametric FEM for Surface Evolution
	\thanks{Submitted to the editors \today.
\funding{The research of the first author was partially supported by Shenzhen Science and Technology Program (grant No.~JCYJ20250604173044021) and the National Natural Science Foundation of China (grant No.~12201418).
        The research of the second author was partially supported by the Fundamental Research Funds for the Central Universities (grant No.~G2025KY05239) and the National Natural Science Foundation of China (grant No.~12401545).}}
}
\author{Beiping Duan\thanks{%
    MSU-BIT-SMBU Joint Research Center of  Applied Mathematics, 
    Shenzhen MSU-BIT University, Shenzhen  518172, P.R. China
    (\email{duanbeiping@smbu.edu.cn}).}
\and Zongze Yang\thanks{%
    School of Mathematics and Statistics,
    Northwestern Polytechnical University, Xi'an, Shaanxi, 710072, P.R. China
    (\email{zongze.yang@nwpu.edu.cn}).}}

\ifpdf
\hypersetup{
  pdftitle={A Second-order Structure-preserving Parametric FEM for surface evolution},
  pdfauthor={Beiping Duan and Zongze Yang}
}
\fi

\begin{document}

\maketitle

\begin{abstract}
    In this paper, we propose a second-order-in-time, structure-preserving, and mesh-robust parametric finite element method for surface diffusion and the volume-preserving mean curvature flow. We first reformulate the original evolution equations into new systems in which the tangential motion is governed by a harmonic map heat flow. This heat flow maps a fixed reference surface onto the unknown evolving surface and drives points on the evolving surface to move in their tangent spaces so as to reduce the associated harmonic energy. As a result, in the discrete setting, the mesh quality can be maintained at a level comparable to that of the reference surface, unless singularities occur. 
    The volume-preserving property is theoretically guaranteed by the delicate design of the scheme, while energy dissipation is enforced through a Lagrange multiplier. We present several numerical experiments to demonstrate the second-order convergence in time and the advantage of the proposed method in preserving mesh quality. The structure-preserving properties are further confirmed by the numerical results. Finally, the proposed framework can be readily extended to other geometric flows.
\end{abstract}



\section{Introduction}

Let $\cM$ denote a closed {\it regular surface} in $\mathbb{R}^{3}$ (see \cite[Section 2-2]{do2016differential} for the definition). Assume that $\vX(t): \cM\rightarrow \Gamma(t)$ is a diffeomorphism that parameterizes the closed evolving surface  $\Gamma(t)$ over $\cM$, then the evolution of $\Gamma(t)$ can be described by the flow map $\vX(t)$:
\begin{equation}\label{SD}
    \frac{\partial \vX(t)}{\partial t}\circ\vX^{-1}(t)\cdot \vnu=v_n
\end{equation}
with $\nu$ the unit outward normal vector and $v_n$ the normal component of the velocity defined on $\Gamma(t)$. In this paper, we focus on surface diffusion and the volume-preserving mean curvature flow. In both cases, we have
\begin{equation*}
    \frac{d}{dt}|\Gamma(t)|\le 0\quad \mbox{and }\int_{\Gamma(t)}v_n \d_{\Gamma}=0,
\end{equation*}
where $|\Gamma(t)|:=\int_{\Gamma(t)}1\;\d_\Gamma$ denotes the surface area of $\Gamma(t)$.
That is, both flows decrease the surface area while preserving the enclosed volume.

Parametric finite element methods (PFEMs) for surface evolution equations originate from Dzuik's pioneering work \cite{Dziuk-1990}. The basic idea of PFEMs is to represent the unknown numerical surface by continuous piecewise polynomials of degree $k\ge1$ parameterized on  reference simplices. The numerical surface is then uniquely determined by the parametrization.
In the case of Lagrangian basis,  we  obtain the positions of all the nodes $\{p_j\}_{j=1}^{N_p}$ after solving the numerical system. However, the nodes may cluster improperly, causing the parametrization to become singular, which can severely degrade accuracy and even terminate the simulation.

One way to avoid mesh distortion (improper clustering of nodes) is to add artificial tangential motions into the original evolution  equation \eqref{SD}. In 2008, by coupling the original evolution equation with the geometric identity $\Delta_{\Gamma(t)}\Id=-H\vnu$, Barrett, Garcke and N\"urnberg \cite{BGN2008} proposed a linear Euler scheme for  mean curvature flow and surface diffusion. Their scheme is energy-decaying, but the enclosed volume is not guaranteed to be conserved for surface diffusion. Based on the Barrett-Garcke-N\"urnberg (BGN)  method, in 2021, Bao and Zhao \cite{Bao-Zhao-2021} constructed a structure-preserving nonlinear Euler scheme for surface diffusion, which means it is energy-decaying and volume-preserving in the discrete sense. Recently, Garcke et~al.~\cite{Garcke-Jiang-2025}  presented a family of second-order-in-time, structure-preserving schemes based on the BGN formulation. The energy-decaying and volume-preserving properties are enforced by two Lagrange multipliers. 

Besides the formulations based on the BGN method, several other approaches have been proposed to introduce artificial tangential motions in PFEMs. In 2014, Mikula et~al.~\cite{mikula2014manifold} introduced the volume-oriented and length-oriented tangential redistribution methods in a general setting of an  $m$-dimensional manifold evolving in an $n$-dimensional manifold. In 2017, Elliott and Fritz \cite{Elliott-Fritz-2017} proposed fully discrete schemes for mean curvature flow based on DeTurck's trick, which can avoid the mesh distortion that may arise in formulations based on the BGN method. In 2022, Hu and Li \cite{hu2022evolving} introduced another strategy to maintain node distribution for evolving surfaces under mean curvature flow and Willmore flow. The idea is to minimize the deformation rate (MDR) at each time step by restricting $\Delta_{\Gamma(t)}\partial_t\vX$ to be parallel to the normal vector field $\vnu$. In 2024, Duan and Li \cite{duan2024new} introduced  a new formulation to induce artificial tangential motions by requiring the flow map $\vX$ to be harmonic between the tangent spaces of $\Gamma(0)$ and $\Gamma(t)$, which is equivalent to minimizing the deformation (MD) between $\Gamma(0)$ and $\Gamma(t)$. Since harmonic maps are conformal or nearly conformal, parametric schemes based on this formulation can maintain the mesh quality of the evolving surface at the same level as that of  the initial mesh. Subsequently, by coupling the original geometric evolution equation with a harmonic map heat flow  from a fixed reference surface $\cM$ onto the unknown surface $\Gamma(t)$, Duan \cite{duan2024energy} proposed a family of  energy-stable schemes for mean curvature flow  that exhibit good mesh quality. Later in 2025, energy-decaying nonlinear Euler schemes for mean curvature flow and surface diffusion were proposed in \cite{duan2025mesh}  based on the reformulation in  \cite{duan2024energy}.

However, none of the above methods preserve the enclosed volume $V$ when applied to surface diffusion. Recently, Gao and Li \cite{gao2025geometric} proposed a structure-preserving approach for surface diffusion based on the MD formulation. Conservation of $V$ is achieved by adopting the strategy in \cite{Bao-Zhao-2021}, while energy decay is enforced through a Lagrange multiplier.
In the most recent work \cite{Gao-Garcke-Li-Tang-2026}, through a delicate analysis of the differences between the BGN method and the MDR method, the authors built a linear energy-decaying Euler schemes for mean curvature flow and surface diffusion that avoid the aforementioned drawback of the BGN method and produce mesh distributions similar to those of the MDR method.

We would also like to mention some related work on error analysis. For example, the skill of matrix formulation developed in \cite{KLLP2017,Kovacs-Li-Lubich-2019,B2021A} applies to finite elements of degree $k\ge 2$, while the trick of dynamic Ritz projection proposed in \cite{li2025dynamic} fills the gap for elements of degree $k=1$. For numerical analysis of curve diffusion, interested readers may refer to the recent work of Deckelnick and N\"urnberg \cite{deckelnick2025finite,deckelnick2025second} on continuous-in-time semidiscrete schemes in arbitrary codimension and fully discrete schemes for curve diffusion in $\mathbb{R}^2$.

Let $\mS$ and $\mS'$ be two smooth and orientable surfaces with or without boundary, and let $\Psi$ be a map from $\mS$ onto $\mS'$. Then straightforward calculations give
\begin{equation}\label{harmonic-ineq}
    \frac{1}{2}\int_{\mS}|\nabla_\mS {\Psi}|^2 \d_\mS\ge |\mS'|, \quad \mbox{and}\quad \frac{1}{2}\int_{\mS}|\nabla_\mS \Id|^2 \d_\mS=|\mS|.
\end{equation}
The first inequality in \eqref{harmonic-ineq}  shows that the harmonic energy (the left-hand side) is always no smaller than the surface area of the target surface $\mS'$, while the second identity implies that the identity map is harmonic.

For geometric flows driven by the surface area functional $|\Gamma(t)|$, essentially the only available technique for designing energy-decaying schemes without using a Lagrange multiplier is based on \eqref{harmonic-ineq}. Specifically, the surface area at $t_{n+1}$ is obtained from the inequality in \eqref{harmonic-ineq}, whereas the area at $t_{n}$ follows from the identity in \eqref{harmonic-ineq}. However, to the best of our knowledge, among time-discretization methods only the Euler scheme can exploit this trick, which leads to first-order accuracy. On the other hand, to construct volume-preserving schemes, e.g., in \cite{gao2025geometric,garcke2023structure}, a commonly used strategy is to adopt the averaged-in-time normal vector proposed in \cite{Bao-Zhao-2021}, which also yields first-order accuracy.

The aim of this paper is to develop second-order-in-time schemes that decay the energy while preserving the enclosed volume. To maintain mesh quality, we adopt the method developed in our previous work \cite{duan2024energy}, where the tangential velocity is determined by a harmonic map heat flow from a reference surface $\cM$ onto the unknown surface $\Gamma(t)$. Since the tangential motion decreases the harmonic energy $\frac{1}{2}\int_{\cM}|\nabla_{\cM}\vX|^2\d_{\cM}$ and harmonic maps are conformal or nearly conformal, the mesh quality can be maintained as the same level as that of the reference surface $\cM$ during the evolution away from singularities. A second-order averaged-in-time normal vectors are utilized to keep the enclosed volume unchanged. Energy decay is enforced through a Lagrange multiplier.

The remainder of this paper is organized as follows. We take surface diffusion as the example to elaborate the way of constructing our structure-preserving scheme in \Cref{Se-SD}. The idea is then generalized to solving the volume-preserving mean curvature flow in \Cref{Se-VCMCF}. Various numerical examples are presented in \Cref{Se-NE} to demonstrate the convergence rate in time, and to verify the structure-preserving and mesh-preserving properties.

\section{Surface diffusion}\label{Se-SD}
In this section we take surface diffusion as the example to illustrate our method. Let  $\cM$ and $\Gamma(t)$ denote the reference surface and the evolving surface, respectively. We use $\vX(t):=\vX(t,\cdot)$ to denote the diffeomorphism from $\cM$ onto $\Gamma(t)$, and $\vnu, H$ to represent the unit outward normal and mean curvature on $\Gamma(t)$. The mean curvature of $\Gamma(t)$ is given by
\begin{equation*}
    H=-\nabla_{\Gamma(t)}\cdot \vnu^T=-\Delta_{\Gamma(t)} \Id\cdot \vnu,
\end{equation*}
where $\nabla_{\Gamma(t)}\cdot \vnu_h^T=\tr(\nabla_{\Gamma(t)}\vnu_h)$. We adopt the convention that the (surface) gradient operator maps a scalar to a row vector. For vector-valued functions, e.g., $\vw=[w_1,w_2,\cdots,w_q]^T$, the $j$th row of $\nabla_{\cM}\vw$ is $\nabla_{\cM} w_j$. For more details of calculus on surfaces, we refer interested readers to \cite{deckelnick2005computation}. Surface diffusion can then be described by the evolution of the flow map $\vX(t,\cdot)$
\begin{equation}\label{SD-1}
    \frac{\partial \vX}{\partial t}\circ\vX^{-1}\cdot \vnu=\Delta_{\Gamma(t)}H.
\end{equation}
\subsection{Weak formulation}
To maintain the mesh quality during the simulation, we adopt the method proposed in our previous work  \cite{duan2024energy} to induce the tangential motions. The idea is to fix the normal velocity and to force the flow map $\vX(t)$ satisfying a modified harmonic map heat flow from a reference surface $\cM$ onto $\Gamma(t)$, given by
\begin{equation}\label{SEEs-HMHF}
         (I-\vnu\otimes\vnu)\frac{\partial \vX}{\partial t}\circ \vX^{-1} = k(t,\vx)(I-\vnu\otimes\vnu)(\Delta_{\cM} \vX) \circ \vX^{-1},
\end{equation}
where $k(t,\vx)$ is a scalar on $\Gamma(t)$ such that $\int_{\Gamma(t)}k(t,\vx) \d_{\Gamma}=\alpha\int_{{\cM}}1 \;\d_{\cM} $ with $\alpha\in \mathbb{R}|_{\ge 0}$ a constant. This heat flow decreases the harmonic energy $\frac{1}{2}\int_{{\cM}}|\nabla_{\cM}\vX|^2\,d_\cM$ measuring the deformation. As a result, it maintains the mesh quality in the discrete sense. 

 Adding \eqref{SD-1} and \eqref{SEEs-HMHF} together, using the curvature expression, we obtain the formulation: Find $\vX(t), H$ and $\lambda(t)\in \mathbb{R}$ with $\vX(0)=\vX^0: \cM\rightarrow \Gamma(0)$ such that
\begin{equation}\label{eq-r-SD}
    \left\{
    \begin{aligned}
        \frac{\partial \vX}{\partial t}\circ \vX^{-1}&=\lambda(\Delta_{\Gamma(t)}H)\vnu + k(t,\vx)(I-\vnu\otimes\vnu)(\Delta_{\cM} \vX) \circ \vX^{-1},\\
        H&= -\Delta_{\Gamma(t)} \Id\cdot \vnu,\\
        \frac{d}{dt}|\Gamma(t)|&=-\int_{\Gamma(t)} |\nabla_{\Gamma(t)}H|^2 \d_\Gamma,
    \end{aligned}
    \right.
\end{equation}
where $\lambda \in \mathbb{R}$ is a Lagrange multiplier. The following lemma reveals that $\lambda\equiv 1$ provided $\Gamma(t)$ is not a sphere. 
\begin{proposition}\label{prop-SDF}
    Suppose there exist sufficiently smooth $\vX(t)$ and $H$ solving \eqref{eq-r-SD} for $t\in (0,T]$, then it follows that $\lambda(t)\equiv 1$ for $t\in (0,T]$ provided that $\Gamma(t)$ is not a sphere.
    \begin{proof}
        Testing the first equation in \eqref{eq-r-SD} with $H\vnu$ yields
        \begin{equation*}
            \int_{\Gamma(t)}\frac{\partial \vX}{\partial t}\circ \vX^{-1} \cdot H\vnu \d_\Gamma =-\lambda\int_{\Gamma(t)}|\nabla_{\Gamma(t)}H |^2\;\d_\Gamma.
        \end{equation*}
    Note that
    \begin{equation*}
        \begin{aligned}
            \frac{d}{dt}|\Gamma(t)|&=\int_{\Gamma(t)}\nabla_{\Gamma(t)}\Id :\nabla_{\Gamma(t)}(\partial_t\vX\circ\vX^{-1})\;\d_{\Gamma}\\
            &=-\int_{\Gamma(t)} \partial_t\vX\circ\vX^{-1}\cdot \Delta_{\Gamma(t)}\Id\; \d_{\Gamma}=\int_{\Gamma(t)} \partial_t\vX\circ\vX^{-1}\cdot H\vnu\; \d_{\Gamma}.    
        \end{aligned}
    \end{equation*}
    Therefore, we have
    \begin{equation*}
        \frac{d}{dt}|\Gamma(t)|=-\lambda\int_{\Gamma(t)}|\nabla_{\Gamma(t)}H |^2\;\d_\Gamma.
    \end{equation*}
Comparing with the last equation in \eqref{eq-r-SD}, we obtain
 \begin{equation}\label{eq-lambda}
     (\lambda-1)\int_{\Gamma(t)}|\nabla_{\Gamma(t)}H |^2\;\d_\Gamma=0.
 \end{equation}
Hence, if $\Gamma(t)$ is not a sphere, the integral in \eqref{eq-lambda} will not vanish, which implies  $\lambda\equiv 1$.
    \end{proof}
\end{proposition}

The corresponding weak formulation reads: Find $\vX(t)\in \mathbf{H}^1(\cM)^3, H\in \mathbf{H}^1(\Gamma(t))$ and $\lambda\in \mathbb{R}$ such that
\begin{equation}\label{weak-formulation}
    \left\{
    \begin{aligned}
        \Big(\partial_t\vX \circ\vX^{-1}, \vxi\Big)_{\Gamma(t)}+\alpha\Big(\gM \vX,  \gM(\mathbb{P}\,\vxi\circ \vX)\Big)_{\cM}&\\
        +\lambda\Big(\nabla_{\Gamma(t)}H,\nabla_{\Gamma(t)}(\vxi\cdot \vnu)\Big)_{\Gamma(t)} &=0, \\
         \Big(\nabla_{\Gamma(t)}\Id,\nabla_{\Gamma(t)}(\eta\vnu)\Big)_{\Gamma(t)}-\Big(H,\eta \Big)_{\Gamma(t)}&=0, \\
        \frac{d}{dt}|\Gamma(t)|+\int_{\Gamma(t)} |\nabla_{\Gamma(t)}H|^2 \d_\Gamma&=0,
    \end{aligned}
    \right.
\end{equation}
holds  $ \forall\vxi\in H^1(\Gamma(t))^3$ and $\eta\in H^1(\Gamma(t))$, where $\mathbb{P}=I-\vnu\otimes\vnu$.
\begin{remark}
    In some situations, for example when using FEM packages such as {\texttt{Firedrake}} \cite{Rathgeber2017Firedrake}, the system must be formulated on a single surface, either $\Gamma(t)$ or $\mathcal{M}$. In the former case, the domain of the term involving $\alpha$ can be transplanted to $\Gamma(t)$ via
    \begin{equation*}
        \Big(\gM \vX,  \gM(\mathbb{P}\,\vxi\circ \vX)\Big)_{\cM}=\Big(\gGamma \Id\, J^{-1},  \gGamma(\mathbb{P}\,\vxi)\sqrt{\det(J)}\Big)_{\Gamma(t)},
    \end{equation*}
where $J=(\gGamma\vX^{-1})^{T}\gGamma\vX^{-1}+\vnu\otimes\vnu$, see e.g., \cite{duan2025mesh}.
\end{remark}
\subsection{Semidiscrete scheme}
Denote by $\{\cM_h\}_h$ an admissible family of shape-regular and quasi-uniform triangulations of $\cM$. For any $\cM_h\in \{\cM_h\}_h$, let $\{\tT_k\}_{k=1}^{N_T}$ be the collocation of simplices and $\{\tp_k\}_{k=1}^{N_p}$ the collocation of vertices. For any polyhedral surface $\cN_h = \cup_{k}T_k$ where $T_k$ are simplices, we introduce the finite element spaces
\begin{equation*}
    S(\cN_h)=\left\{w|\; w \mbox{ is continuous  and piecewise linear on each simplex} \right\}
\end{equation*}
and 
\begin{equation*}
    W(\cN_h)=\left\{w|\;w \mbox{ is constant on each simplex}\right\}.
\end{equation*}
In addition,  we define
\begin{equation}\label{eq-J}
    \mJ: \cN_h \rightarrow W(\cN_h)^3,\quad \mbox{with } \mJ(T)|_{T\in \{T_k\}}=(p_2-p_1)\times (p_3-p_1),
\end{equation}
where $\{p_j\}_{j=1}^3$, ordered counterclockwise  when viewed from the outward normal direction, are the vertices of the simplex $T$.

Denote by $\vX_h(t)\in S(\cM_h)^3$ the parameterization of $\Gamma_h(t)$ consisting of simplices $\{T_k(t)\}_{k=1}^{N_T}$ with vertices $\{p_k(t)\}_{k=1}^{N_p}$.
For any two functions $f_h$, $g_h$ defined on $\cN_h$ which are continuous on each simplex but may have jumps across the edges, we use $(\cdot,\cdot)_{\cN_h}^h$ to represent the mass lumped inner product on $\cN_h$, given by
\begin{equation*}
(f_h, g_h)_{\cN_h}^h=\frac{1}{3}\sum_{T\in \cN_h}|T|\sum_{j=1}^3 f_h(p_j)g_h(p_j), 
\end{equation*}
where $\{p_j\}_{j=1}^3$ are the vertices of $T$.

For the polyhedral surface $\Gamma_h(t)$, let $\pi_h$ denote the $L^2$-projection from $W(\Gamma_h(t))^3$ onto $S(\Gamma_h(t))^3$, defined by 
\begin{equation*}
    \left(\pi_h\vec{\chi}_h, \vxi_h\right)_{\Gamma_h(t)}^h=\left(\vec{\chi}_h, \vxi_h\right)_{\Gamma_h(t)}, \quad \forall\vxi_h\in S(\Gamma_h(t))^3.
\end{equation*}
The nodal values of $\pi_h\vec{{\chi}}_h$ can be written explicitly as
\begin{equation}\label{def-vmu0}
    (\pi_h\vec{\chi}_h)(p_j)=\frac{\sum_{k,p_j\in T_k} |T_k|\vec{\chi}_{h,k}}{{\sum_{k,p_j\in T_k} |T_k|}},
    \quad \vec{\chi}_{h,k}=\vec{\chi}_h|_{T_k}.
\end{equation}
Let $\vn_{h}\in W({\Gamma_h(t)})^3$ denote the unit outward normal vector of $\Gamma_h(t)$ and set $\vn_{h,k}=\vn_h|_{T_k}$, where we omitted the dependence of time for simplicity. We further define the weighted outward normal vector $\vnu_h\in S({\Gamma_h(t)})^3$ by 
\begin{equation}\label{def-vnu0}
    {\vnu}_{h}(p_j):=\frac{\pi_h\vn_h(p_j)}{|\pi_h\vn_h(p_j)|}=\frac{\sum_{k,p_j\in T_k} |T_k|\vn_{h,k}}{\left|\sum_{k,p_j\in T_k} |T_k|\vn_{h,k}\right|}.
\end{equation}

The semidiscrete scheme in space reads: Find $\vX_h(t)\in S(\cM_h)^3$, $H_h\in S(\Gamma_h(t))$ and $\lambda_h\in \mathbb{R}$ such that 
\begin{equation}\label{semi-scheme}
    \left\{
\begin{aligned}
    \Big(\partial_t\vX_h \circ\vX_h^{-1}, \vxi_h\Big)_{\Gamma_h(t)}^h+\alpha\Big(\gMh \vX_h,  \gMh\cI_h(\mathbb{P}_{h}\vxi_h\circ \vX_h)\Big)_{\cM_h} &\\
    +\lambda_h\Big(\nabla_{\Gamma_h(t)}H_h,\nabla_{\Gamma_h(t)}\cI_h\big(|\pi_h\vn_h|^{-1}\vxi_h\cdot\vnu_h\big)\Big)_{\Gamma_h(t)}&=0, \\
    \Big(\nabla_{\Gamma_h(t)}\Id,\nabla_{\Gamma_h(t)}\cI_h(\eta_h\vnu_h)\Big)_{\Gamma(t)}-\Big(H_h,\eta_h\Big)^h_{\Gamma_h(t)}&=0, \\
    \frac{d}{dt}|\Gamma_h(t)|+\int_{\Gamma_h(t)}|\nabla_{\Gamma_h(t)}H_h|^2 \;\d_{\Gamma_h}&=0,
\end{aligned}
\right.
\end{equation}
holds  for all $ \vxi_h\in H^1(\Gamma_h(t))^3$ and $\eta_h\in H^1(\Gamma_h(t))$, where $\mathbb{P}_{h}=\vnu_h\cdot \pi_h\vn_h-(\pi_h\vn_h) \vnu^T_h$ and $\cI_h$ denotes the piecewise linear interpolation operator.
\begin{remark}
    The standard $L^2$-inner product $(H_h, \eta_h)_{\Gamma_h(t)}$  also can be used instead of the mass-lumped one $(H_h, \eta_h)_{\Gamma_h(t)}^h$ in \eqref{semi-scheme} and in the fully discrete scheme, as well as in the corresponding formulations for the volume-preserving mean curvature flow. Mass lumping only affects the sparsity of the resulting system.
\end{remark}
\begin{theorem}
    Let $\Gamma_h(t)$ be the solution of \eqref{semi-scheme} and let $V_h(t)$ denote the  volume  enclosed by $\Gamma_h(t)$. Then $\frac{d}{dt}V_h(t)=0$.
\end{theorem}
\begin{proof}
    Taking $\vxi_h=\pi_h\vn_h$ in \eqref{semi-scheme} with noting that 
    \begin{equation*}
        \begin{aligned}
            \mathbb{P}_h\pi_h\vn_h&=\big(\vnu_h\cdot \pi_h\vn_h-(\pi_h\vn_h) \vnu^T_h\big)\pi_h\vn_h\\
            &=(\vnu_h\cdot \pi_h\vn_h)\pi_h\vn_h-\pi_h\vn_h(\vnu^T_h\pi_h\vn_h)\\
            &=\vec{0}
        \end{aligned}
    \end{equation*}
and $\cI_h\big(|\pi_h\vn_h|^{-1}\pi_h\vn_h\cdot\vnu_h\big)=\cI_h(\vnu_h\cdot\vnu_h)=1$, we obtain 
\begin{equation*}
    \Big(\partial_t\vX_h \circ\vX_h^{-1}, \pi_h\vn_h\Big)_{\Gamma_h(t)}^h=\Big(\partial_t\vX_h \circ\vX_h^{-1}, \vn_h\Big)_{\Gamma_h(t)}=\frac{d}{dt}V_h(t)=0,
\end{equation*}
where the first step follows from the definition of $\pi_h$ and the second step follows from the transport theorem.
\end{proof}
\subsection{Fully discrete scheme}
Let $0=t_0<t_1<\cdots<t_{M-1}<t_M=T$ be a quasi-uniform partition  of $[0,T]$ and denote by $\tau_m=t_{m}-t_{m-1}$ the time step size. Let $\Gamma_h(t_m)$ represent the image of $\vX_h(t_m)$ and  $\hat{\Gamma}_h(t_{m+\half})=\cup_{k}\hat{T}_k(t_{m+\half})$ the image of the extrapolation $\hX_h(t_{m+\half}):=\vX_h(t_m)+\frac{\tau_{m+1}}{2\tau_m}\big(\vX_h(t_m)-\vX_h(t_{m-1})\big) (m\ge 1)$. To introduce our second order scheme, we define an averaged-in-time outward normal vector $\tn_h(t_{m+\half})$ for $\Gamma_h(t_{m+\half})$:
\begin{equation}\label{modify-n}
    \tn_{h,k}(t_{m+\half}):=\tn_h(t_{m+\half})\big|_{T_k(t_{m+\half})}=\frac{\mJ_k(t_m)+4\mJ_k(t_{m+\half})+\mJ_k(t_{m+1}) }{6|\hat{\mJ}(t_{m+\half})|}
\end{equation}
where $\mJ_k(t):=\mJ(T_k(t))$ and $\hat{\mJ}_k(t_{m+\half})=\mJ(\hat{T}_k(t_{m+\half}))$. This normal is formally a second-order approximation to $\vn_{h,k}$. In fact, Taylor expansion gives
\begin{equation*}
\begin{aligned}
    &\tn_{h,k}(t_{m+\half})-\vn_{h,k}(t_{m+\half})\\
    &=\frac{\mJ_k(t_m)+4\mJ_k(t_{m+\half})+\mJ_k(t_{m+1}) }{6|\hat{\mJ}_k(t_{m+\half})|}-\frac{\mJ_k(t_{m+\half})}{|\hat{\mJ}_k(t_{m+\half}|}+\frac{\mJ_k(t_{m+\half})}{|\hat{\mJ}_k(t_{m+\half}|}-\frac{\mJ_k(t_{m+\half})}{|\mJ_k(t_{m+\half}|}\\
    &=\frac{\mJ_k(t_m)+\mJ_k(t_{m+1})-2\mJ_k(t_{m+\half}) }{6|\hat{\mJ}(t_{m+\half})|}+\frac{\mJ_k(t_{m+\half})}{|\hat{\mJ}_k(t_{m+\half}|}-\frac{\mJ_k(t_{m+\half})}{|\mJ_k(t_{m+\half}|}\\
 &= \mO(\tau_{m+1}^2)  +\vn_{h,k}\mO\big(\tau_m^2+\tau_{m+1}^2\big).
\end{aligned}
\end{equation*}
For $m=0$, since no extrapolation can be used, we define $\tn_{h}(t_\half)$ by
\begin{equation}\label{modify-n-0}
    \tn_{h,k}(t_{\half}):=\tn_h(t_{\half})\big|_{T_k(t_{\half})}=\frac{\mJ_k(t_0)+4\mJ_k(t_{\half})+\mJ_k(t_{1}) }{6|{\mJ}(t_{\half})|}.
\end{equation}

Denote by $\Gamma_h^{m}=\cup_k{T_k^m}$ the polyhedral surface produced by the numerical scheme, where $\{p_{k,j}^m\}_{j=1}^3$ (ordered counterclockwise) are the vertices of the simplex $T_k^m$. Let $\vX_h^m\in S(\cM_h)^3$ be the parameterization of $\Gamma_h^m$, and 
\begin{equation*}
\hX_h^{m+\half}:=\vX_h^m+\frac{\tau_{m+1}}{2\tau_m}(\vX_h^{m}-\vX_h^{m-1})\quad m\ge 1
\end{equation*}
be the linear extrapolation whose image is $\hGam=\cup_k\hat{T}_k^{m+\half}$.  For $m\ge 0$ we further denote $\Gamma_h^{m+\half}=\cup_k {T}_k^{m+\half}$ which is the image of the flow map $\vX_h^{m+\half}:=\frac{1}{2}(\vX_h^{m}+\vX_h^{m+1})$. We further let $\hnu_h^{m+\half}\in S(\hGam)^3$ represent the weighted unit outward normal vector of $\hGam$. 

For $m\ge 1$, the fully discrete scheme reads: Find $\vu_h^{m+1}\in S(\hGam)^3, H_h^{m+\half}\in S(\hGam)$ and $\lambda_h^{m+1}\in \mathbb{R}$ such that 
\begin{equation}\label{CN-scheme}
        \left\{
\begin{aligned}
    \left(\frac{\vu_h^{m+1}-\vu_h^m}{\tau_{m+1}},\vxi_h\right)^h_{\hGam} +\alpha\Big(\gMh \vX_h^{m+\half},
    \gMh\cI_h\big(\mathbb{P}_{h}^{m+\half}\vxi_h\circ \hX_h^{m+\half}\big)\Big)_{\cM_h}&\\
        +\,\lambda_h^{m+1}\left(\nabla_{\hGam} H_h^{m+\half},
            \nabla_{\hGam}\cI_h\left[
                \frac{\vxi_h\cdot\pi_h\tn_h^{m+\half}}{|\pi_h\tn_h^{m+\half}|^2}
            \right]\right)_{\hGam}=0&,\\
    \left(\nabla_{\hGam}\frac{\vu_h^{m+1}+\vu_h^m}{2}, \nabla_{\hGam}\cI_h(\eta_h\vnu_h^{m+\half})\right)_{\hGam}
        =\left(H_h^{m+\half},\eta_h\right)^h_{\hGam},&\\
    \frac{|\Gamma_h^{m+1}|-|\Gamma_h^m|}{\tau_{m+1}}
        +\left(\nabla_{\hGam}H_h^{m+\half},\nabla_{\hGam}H_h^{m+\half}\right)_{\hGam}=0,&
\end{aligned}
    \right.
\end{equation}
holds for $\forall \vxi_h\in S(\hGam)^3$ and $\eta_h\in S(\hGam)$, where, for simplicity, we adopted the notation $\vu_h^m:=\vX_h^{m}\circ(\hX_h^{m+\half})^{-1}$ and denoted by $\vnu_h^{m+\half}$ the weighted unit outward normal of $\Gamma_h^{m+\half}$ lifted on $\hGam$. The projection operator $\mathbb{P}_{h}^{m+\half}$ is given by
\begin{equation}\label{def-P}
    \mathbb{P}_{h}^{m+\half}=\hnu_h^{m+\half}\cdot \pi_h\tn_h^{m+\half}-(\pi_h\tn_h^{m+\half})(\hnu_h^{m+\half})^T
\end{equation}
with $\tn_h^{m+\half}$ defined by 
\begin{equation*}
    \tn_h^{m+\half}\big|_{\hat{T}_{k}^{m+\half}}=\frac{\mJ_k^{m}+4\mJ_k^{m+\half}+\mJ_k^{m+1}}{6|\hat{\mJ}_k^{m+\half}|}\qquad m\ge 1, 
\end{equation*}
where $\mJ_k^m=\mJ(T_k^m)$, $\mJ_k^{m+\half}=\mJ(T_k^{m+\half})$ and $\hat{\mJ}_k^{m+\half}=J(\hat{T}_k^{m+\half})$.

For $m=0$, we adopt a fully nonlinear scheme in which $\hat{\Gamma}_h^\half$ is replaced by the unknown $\Gamma_h^\half$. The scheme reads: Find $\vu_h^1\in S(\Gamma_h^\half)^3$, $H_h^\half\in S(\Gamma_h^\half)$ and $\lambda_h^1\in \mathbb{R}$ such that 
\begin{equation}\label{CN-scheme-start}
    \left\{
    \begin{aligned}
        \left(\frac{\vu_h^{1}-\vu_h^0}{\tau_{1}} ,\vxi_h\right)^h_{\Gamma_h^\half}+ \lambda_h^1\left(\nabla_{\Gamma_h^\half} H_h^{\half},\nabla_{\Gamma_h^\half}\cI_h\left[\frac{\vxi_h\cdot\pi_h\tn_h^{\half}}{|\pi_h\tn_h^{\half}|^2}\right]\right)_{\Gamma_h^\half}&\\
        +\,\alpha\Big(\gMh \vX_h^{\half},  \gMh\cI_h\big(\mathbb{P}_h^{\half}\vxi_h\circ \vX_h^{\half}\big)\Big)_{\cM_h}&=0, \\
        \frac{1}{2}\left(\nabla_{\Gamma_h^\half}(\vu_h^{1}+\vu_h^0), \nabla_{\Gamma_h^\half}\cI_h(\eta_h\vnu_h^{\half})\right)_{\Gamma_h^\half}
        -\left(H_h^{\half},\eta_h\right)^h_{\Gamma_h^\half}&=0 ,\\
        \frac{|\Gamma_h^{1}|-|\Gamma_h^0|}{\tau_{1}}+\left(\nabla_{\Gamma_h^\half}H_h^{\half},\nabla_{\Gamma_h^\half}H_h^{\half}\right)_{\Gamma_h^\half}&=0,
    \end{aligned}
    \right.
\end{equation}
for $\forall \vxi_h\in S(\Gamma_h^\half)^3$ and $\eta_h\in S(\Gamma_h^\half)$. In the scheme above 
\begin{equation}\label{tn-start}
    \tn_h^{\half}\big|_{{T}_{k}^{\half}}=\frac{\mJ_k^{0}+4\mJ_k^{\half}+\mJ_k^{1}}{6|{\mJ}_k^{\half}|}, 
\end{equation}
and $\mathbb{P}_{h}^{\half}=\vnu_h^{\half}\cdot \pi_h\tn_h^{\half}-(\pi_h\tn_h^{\half})(\vnu_h^{\half})^T$ with $\vnu_h^\half\in S(\Gamma_h^\half)^3$  the weighted outward unit nodal normal of $\Gamma_h^\half$.

\begin{theorem}\label{thm-Vol-cons}
    Let $V_h^{m}$ denote the  volume enclosed by $\Gamma_h^{m}$. Then scheme \eqref{CN-scheme} satisfies $V_h^{m+1}=V_h^{m}=\cdots=V_h^1=V_h^0$ and $|\Gamma_h^{m+1}|\le |\Gamma_h^{m}|\le \cdots\le|\Gamma_h^1|\le |\Gamma_h^0|$.
\end{theorem}
\begin{proof}
    Denote $\vz(\theta):=\vz(\theta,\cdot)=(1-\theta)\vu_h^m+\theta\vu_h^{m+1}$ for $\theta\in[0,1]$. Let $\mathbf{u}^m$ be the collocations of  the vertexes of $\Gamma_h^m$ and  we use $\Xi(\theta)$ to represent the polyhedral surface with nodal positions $\mbfz(\theta):=(1-\theta)\mbfu^m+\theta\mbfu^{m+1}$. Let $\vn_h(\theta)\in W(\Xi(\theta))^3$ be the unit outward normal vector of $\Xi(\theta)$. Suppose $\Xi(\theta)=\cup_{k}T_k(\theta)$ and $\{p_{k,j}(\theta)\}_{j=1}^3$ are the nodal positions of $T_k$ located counterclockwise. Denote $\mJ_k(\theta)=(p_{k,2}-p_{k,1})\times(p_{k,3}-p_{k,1})$ and $V(\theta)$ the enclosed volume by $\Xi(\theta)$, then
        \begin{equation*}
        \begin{aligned}
            \frac{d}{d\theta}V(\theta)&=\int_{S(\theta)}\frac{d}{d\theta}\vz(\theta)\circ\vz^{-1}(\theta)\cdot\vn_h(\theta) \,\d_{\Xi(\theta)}\\
            &=\sum_{k}\int_{T_k(\theta)}\frac{d}{d\theta}\vz(\theta)\circ\vz^{-1}(\theta)\cdot \frac{\mJ_{k}(\theta)}{|\mJ_{k}(\theta)|} \,\d_{\Xi(\theta)}.
        \end{aligned}
    \end{equation*}
Utilizing the transplant 
\begin{equation*}
    \int_{T_k(\theta)}|\mJ_k(\theta)|^{-1}\,\d_{\Xi(\theta)}=\int_{\hat{T}_h^{m+\half}}|\hat{\mJ}_{k}^{m+\half}|^{-1}\,\d_{\Gamma_h}
\end{equation*}
and the definition of $\vz(\theta)$ it follows
\begin{equation*}
    \begin{aligned}
        \frac{d}{d\theta}V(\theta)    &=\sum_{k}\int_{\hat{T}_k^{m+\half}}(\vu_h^{m+1}-\vu_h^m)\cdot\frac{\mJ_{k}(\theta)}{|\hat{\mJ}_{k}^{m+\half}|} \,\d_{\Gamma_h}.
    \end{aligned}
\end{equation*}
    Note that $\mJ_k(\theta)$ is a quadratic polynomial with respect to $\theta$. Hence, applying Simpson's rule we get
    \begin{equation}\label{Vol-cons}
        \begin{aligned}
            V(1)-V(0)&=\sum_{k} \int_0^1 \int_{\hat{T}_k^{m+\half}}(\vu_h^{m+1}-\vu_h^m) \cdot\frac{\mJ_{k}(\theta)}{|\hat{\mJ}_{k}^{m+\half}|} \,\d_{\Gamma_h} \,d\theta\\
            &=\sum_{k}\int_{\hat{T}_k^{m+\half}}(\vu_h^{m+1}-\vu_h^m) \cdot\frac{\mJ_{k}(0)+4\mJ_k(\half)+\mJ_k(1)}{6|\hat{\mJ}_{k}^{m+\half}|}\,\d_{\Gamma_h}\\
            &=\int_{\hGam}(\vu_h^{m+1}-\vu_h^m) \cdot\tn_h^{m+\half}\,\d_{\Gamma_h}.
        \end{aligned}
    \end{equation}
Take $\vxi=\pi_h\tn_h^{m+\half}$ in the first equation in \eqref{CN-scheme}. Noting that 
\[\cI_h(\mathbb{P}_h^{m+\half}\pi_h\tn_h^{m+\half})= \vec{0}\]
and
\[\cI_h\big( \pi_h\tn_h^{m+\half}\cdot\pi_h\tn_h^{m+\half}\,{\big|\pi_h\tn_h^{m+\half}\big|^{-2}}\big)=1,\]
it follows that
\begin{equation*}
    \left(\vu_h^{m+1}-\vu_h^m,\pi_h\tn_h^{m+\half}\right)_{\hGam}^h=0.
\end{equation*}
Appealing to the definition of $\pi_h$ we get
\begin{equation*}
    \left(\vu_h^{m+1}-\vu_h^m,\tn_h^{m+\half}\right)_{\hGam}=0.
\end{equation*}
Therefore,$V(1)-V(0)=0$, which implies $V_h^{m+1}=V_h^{m}$ for $m\ge 1$. For $m=0$, take $\vxi_h=\pi_h\tn_h^{\half}$ in the first equation of \eqref{CN-scheme-start} and   repeat the procedures in \eqref{Vol-cons} with $\hat{\Gamma}_h^\half$  replaced by $\Gamma_h^\half$ and $\tn_h^\half$ defined in \eqref{tn-start}, then one can get $V_h^1=V_h^0$. The energy decay follows directly from the last equations in \eqref{CN-scheme} and \eqref{CN-scheme-start}.
\end{proof}
\begin{remark}
    Empirically, Euler scheme is often used to initialize a second-order method involving standard extrapolation. For example, one may employ the following Euler scheme at the first time step
    \begin{equation}\label{Euler-scheme}
        \left\{
        \begin{aligned}
            \left(\frac{\vu_h^{1}-\vu_h^0}{\tau_{1}} ,\vxi_h\right)^h_{\Gamma_h^0}+ \lambda_h^1\left(\nabla_{\Gamma_h^0} H_h^{1},\nabla_{\Gamma_h^0}\cI_h\left[\frac{\vxi_h\cdot\pi_h\tn_h^{1}}{|\pi_h\tn_h^{1}|^2}\right]\right)_{\Gamma_h^0}&\\
            +\alpha\Big(\gMh \vX_h^{1},  \gMh\cI_h(\mathbb{P}_h^1\vxi_h\circ \vX_h^0)\Big)_{\cM_h}=0&, \quad \forall \vxi_h\in S(\Gamma_h^0)^3\\
            \left(\nabla_{\Gamma_h^0}\vu_h^{1}, \nabla_{\Gamma_h^0}\cI_h(\eta_h\vnu_h^{0})\right)_{\Gamma_h^0}
            -\left(H_h^{1},\eta_h\right)^h_{\Gamma_h^0}=0&, \quad \forall \eta_h\in S(\Gamma_h^0)\\
            \frac{|\Gamma_h^{1}|-|\Gamma_h^0|}{\tau_1}+\left(\nabla_{\Gamma_h^0}H_h^{1},\nabla_{\Gamma_h^0}H_h^{1}\right)_{\Gamma_h^0}=0&,
        \end{aligned}
        \right.
    \end{equation}
    where 
    \begin{equation}\label{Euler-n-P}
        \tn_h^{1}\big|_{{T}_{k}^{0}}=\frac{\mJ_k^{0}+4\mJ_k^{\half}+\mJ_k^{1}}{6|{\mJ}_k^{0}|} \quad \mbox{and}\quad \mathbb{P}_{h}^{1}=\vnu_h^{0}\cdot \pi_h\tn_h^{1}-(\pi_h\tn_h^{1})(\vnu_h^{0})^T.
    \end{equation}
Like before, one can verify easily that $V_h^1=V_h^0$ holds. However, numerical experiments (see \Cref{con-t-ex1}) indicate that the temporal convergence rate in time is reduced if this Euler scheme is used.
\end{remark}

\subsection{Iterative method}
In practice, we employ the following simple and effective iterative method to solve \eqref{CN-scheme}. Let $(\vu_{h,l}^{m+1},H_{h,l}^{m+\half},\lambda_{h,l}^{m+1})$ denote the solution after the $l$th iteration with $\Gamma_{h,l}^{m+1}$ the corresponding numerical surface given by $\vX_{h,l}^{m+1}:=\vu_{h,l}^{m+1}\circ\hat{X}_{h}^{m+\half}$. We choose the initial guess $(\vu_{h,0}^{m+1},H_{h,0}^{m+\half},\lambda_{h,0}^{m+1})=(\vu_{h}^{m},H_{h}^{m-\half},\lambda_{h}^{m})$ and update the solution by solving the following linear system
 \begin{equation}\label{CN-scheme-linear}
     \left\{
    \begin{aligned}
        \left(\frac{\vu_{h,l+1}^{m+1}-\vu_h^m}{\tau_{m+1}} ,\vxi_h\right)^h_{\hGam} +\alpha\Big(\gMh \vX_{h,l+1}^{m+\half},
        \gMh\cI_h\big(\mathbb{P}_{h,l}^{m+\half}\vxi_h\circ \hX_h^{m+\half}\big)\Big)_{\cM_h}& \\
       +(\lambda_{h,l+1}^{m+1}-\lambda_{h,l}^{m+1})\left(
                \nabla_{\hGam} H_{h,l}^{m+\half},
                \nabla_{\hGam}\cI_h\left[
                    \frac{\vxi_h\cdot\pi_h\tn_{h,l}^{m+\half}}{|\pi_h\tn_{h,l}^{m+\half}|^2}
                \right]\right)_{\hGam}& \\
                  +\lambda_{h,l}^{m+1}\left(\nabla_{\hGam} H_{h,l+1}^{m+\half},
                \nabla_{\hGam}\cI_h\left[
                \frac{\vxi_h\cdot\pi_h\tn_{h,l}^{m+\half}}{|\pi_h\tn_{h,l}^{m+\half}|^2}
                \right]\right)_{\hGam}=0&,  \\
        \left(\nabla_{\hGam}\frac{\vu_{h,l+1}^{m+1}+\vu_h^m}{2}, \nabla_{\hGam}\cI_h(\eta_h\vnu_{l}^{m+\half})\right)_{\hGam}
            =\left(H_{h,l+1}^{m+\half},\eta_h\right)^h_{\hGam}&, \\
        \frac{|\Gamma_{h,l}^{m+1}|-|\Gamma_h^m|}{\tau_{m+1}}
            +\left(\nabla_{\hGam}H_{h,l}^{m+\half},\nabla_{\hGam}H_{h,l}^{m+\half}\right)_{\hGam} \hspace{2.5cm}& \\
        +\frac{1}{\tau_{m+1}}\int_{\Gamma_{h,l}^{m+1}}\nabla_{\Gamma_{h,l}^{m+1}}\cdot\big(\vu_{h,l+1}^{m+1}-\vu_{h,l}^{m+1}\big)\,\d_{\Gamma_h}\hspace{1cm}&\\
        +2\left(\nabla_{\hGam}(H_{h,l+1}^{m+\half}-H_{h,l}^{m+\half}),\nabla_{\hGam}H_{h,l}^{m+\half}\right)_{\hGam}=0&, 
    \end{aligned}
     \right.
 \end{equation}
where
\begin{equation*}
    \tn_{h,l}^{m+\half}\big|_{\hat{T}_{k}^{m+\half}}=\frac{\mJ_k^{m}+4\mJ_{k,l}^{m+\half}+\mJ_{k,l}^{m+1}}{6|\hat{\mJ}_{k}^{m+\half}|},
\end{equation*}
$\mJ_{k,l}^{m+\half}$ and $\mJ_{k,l}^{m+1}$ correspond to $\mJ(\Gamma_{h,l}^{m+\half})$ and $\mJ(\Gamma_{h,l}^{m+1})$, respectively, and  $\Gamma_{h,l}^{m+\half}$ denotes the image of $\frac{1}{2}(\vX_{h,l}^{m+1}+\vX_h^{m})$ with $\vnu_{h,l}^{m+\half}$ the weighted unit outward normal.  For $m=0$ we apply a similar procedure, where the integral domain is replaced by the image of $\vX_{h,l}^{\half}$ at the $(l+1)$th iteration and the initial domain is taken as $\Gamma_h^0$.  We stop the iteration and set $(\vu_h^{m+1},H_h^{m+\half},\lambda_h^{m+1})=(\vu_{h,l+1}^{m+1},H_{l+1}^{m+\half},\lambda_{h,l+1}^{m+1})$ if the stopping criterion $\max |\vu_{h,l+1}^{m+1}-\vu_{h,l}^{m+1}|\le tol$ is satisfied for a prescribed tolerance $tol$.

\section{Volume-preserving mean curvature flow}\label{Se-VCMCF}
The volume-preserving mean curvature flow originates from the work of Gage \cite{Gage1986} for planar curves in 1986, and Huisken \cite{Huisken1987} for hypersurfaces in 1987. Since then it has been widely studied including the existence and uniqueness, the behavior of the solution, and the connections with diffusion interface models, see e.g., \cite{mugnai2016global,rubinstein1992nonlocal,Mellet2025,Bronsard1997} and the references therein.  The volume-preserving mean curvature flow can be interpreted as the $L^2$-gradient flow induced by the  functional of surface area under unchanged enclosed volume, given by 
 \begin{equation}\label{VCMCF}
     \frac{\partial \vX}{\partial t}\circ\vX^{-1} \cdot \vnu=-\big(H-\langle H\rangle\big),
 \end{equation}
where  $\langle H\rangle:=\frac{1}{|\Gamma(t)|}\int_{\Gamma(t)} H \d_{\Gamma}$ is the averaged mean curvature over $\Gamma(t)$. In the equation above we continued to use $\cM$ and $\Gamma(t)$ to denote the reference surface and the evolving surface, respectively, and $\vX$ to denote the flow map from $\cM$ onto $\Gamma(t)$.   Like before, we couple \eqref{VCMCF}  with the harmonic map heat flow and introduce a Lagrange multiplier for the energy then reformulate it into the following system
\begin{equation}\label{eq-r-VCMCF}
    \left\{
    \begin{aligned}
        \frac{\partial \vX}{\partial t}\circ \vX^{-1}&=-\lambda\big(H-\langle H\rangle\big)\vnu + k(t,\vx)(I-\vnu\otimes\vnu)(\Delta_{\cM} \vX) \circ \vX^{-1},\\
        H&= -\Delta_{\Gamma(t)} \Id\cdot \vnu,\\
        \frac{d}{dt}|\Gamma(t)|&=-\int_{\Gamma(t)} |H-\langle H\rangle|^2 \,\d_\Gamma.
    \end{aligned}
    \right.
\end{equation}
\begin{proposition}\label{prop-MCF}
    Suppose there exist sufficiently smooth $\vX$ and $H$ solving \eqref{eq-r-VCMCF}. Then $\lambda\equiv 1$ provided that $\Gamma(t)$ is not a sphere.
\end{proposition}
\begin{proof}
    The first two equations in \eqref{eq-r-VCMCF}  imply
    \begin{equation*}
        \frac{d}{dt}|\Gamma(t)|=-\lambda\int_{\Gamma(t)} |H-\langle H\rangle|^2 \,\d_\Gamma.
    \end{equation*}
Comparing with the third equation in \eqref{eq-r-VCMCF}, we obtain
\begin{equation*}
    (\lambda-1)\int_{\Gamma(t)} |H-\langle H\rangle|^2 \,\d_\Gamma=0.
\end{equation*}
Hence $\lambda\equiv 1$ if $H$ is not constant on $\Gamma(t)$, which is equivalent to $\Gamma(t)$  not being a sphere.
\end{proof}

The corresponding weak formulation reads: Find  $\vX(t)\in \mathbf{H}^1(\cM)^3, H\in \mathbf{H}^1(\Gamma(t))$ and $\lambda\in \mathbb{R}$ such that
\begin{equation}\label{weak-formulation-VCMCF}
    \left\{
    \begin{aligned}
        \Big(\partial_t\vX \circ\vX^{-1}, \vxi\Big)_{\Gamma(t)}+\lambda\Big(H-\langle H\rangle, \vxi\cdot \vnu\Big)_{\Gamma(t)}\qquad &\\
        +\alpha\Big(\gM \vX,  \gM(\mathbb{P}\,\vxi\circ \vX)\Big)_{\cM}&=0, \\
        \Big(\nabla_{\Gamma(t)}\Id,\nabla_{\Gamma(t)}(\eta\vnu)\Big)_{\Gamma(t)}-\Big(H,\eta \Big)_{\Gamma(t)}&=0, \\
        \frac{d}{dt}|\Gamma(t)|+\int_{\Gamma(t)} |H-\langle H\rangle|^2 \,\d_\Gamma&=0,
    \end{aligned}
    \right.
\end{equation}
holds  $ \forall\vxi\in H^1(\Gamma(t))^3$ and $\eta\in H^1(\Gamma(t))$.

\subsection{Fully discrete scheme}
In the following we adopt the same notation as that used for surface diffusion. Suppose the approximate solution $\Gamma_h^m(m\ge 1)$   is known. Then the numerical scheme for computing $\Gamma_h^{m+1}$ reads: Find $\vu_h^{m+1}\in S(\hGam)^3, H_h^{m+\half}\in S(\hGam)$ and $\lambda_h^{m+1}\in \mathbb{R}$ such that 
\begin{equation}\label{CN-scheme-VCMCF}
    \left\{
    \begin{aligned}
        \left(\frac{\vu_h^{m+1}-\vu_h^m}{\tau_{m+1}} ,\vxi_h\right)^h_{\hGam}+\alpha\Big(\gMh \vX_h^{m+\half},
        \gMh\cI_h\big(\mathbb{P}_{h}^{m+\half}\vxi_h\circ \hX_h^{m+\half}\big)\Big)_{\cM_h}&\\
            +\lambda_h^{m+1}\left(
                H_h^{m+\half}-\langle {H}_h^{m+\half}\rangle,\cI_h\left[
                    \frac{\vxi_h\cdot\pi_h\tn_h^{m+\half}}{|\pi_h\tn_h^{m+\half}|^2}\right]\right)_{\hGam}=0,&\\
        \left(\nabla_{\hGam}\frac{\vu_h^{m+1}+\vu_h^m}{2}, \nabla_{\hGam}(\eta_h\vnu_h^{m+\half})\right)_{\hGam}
            =\left(H_h^{m+\half},\eta_h\right)^h_{\hGam}& ,\\
        \frac{|\Gamma_h^{m+1}|-|\Gamma_h^m|}{\tau_{m+1}}
            + \left(H_h^{m+\half}-\langle H_h^{m+\half}\rangle ,H_h^{m+\half}-\langle H_h^{m+\half}\rangle \right)_{\hGam}=0,&
    \end{aligned}
    \right.
\end{equation}
holds $\forall \vxi_h\in S(\hGam)^3$ and $\eta_h\in S(\hGam)$. In the formulation above, $\langle H_h^{m+\half}\rangle$ is the averaged value of the curvature $H_h^{m+\half}$ over $\hGam$, given by 
\begin{equation*}
    \langle H_h^{m+\half}\rangle:=\frac{1}{|\hGam|}\int_{\hGam} H_h^{m+\half} \,\d_{\Gamma_h}.
\end{equation*}

For $m=0$, similar to \eqref{CN-scheme-start}, we replace $\hat{\Gamma}_h^{\half}$ by the unknown $\Gamma_h^{\half}$ to be determined by the iteration.
\begin{theorem}
    Let $V_h^{m}$ denote the volume enclosed by $\Gamma_h^{m}$. Then scheme \eqref{CN-scheme-VCMCF} satisfies $V_h^{m+1}=V_h^{m}=\cdots=V_h^1=V_h^0$ and $|\Gamma_h^{m+1}|\le |\Gamma_h^{m}|\le \cdots\le |\Gamma_h^1|\le |\Gamma_h^0|$.
\end{theorem}
\begin{proof}
    Take $\vxi_h=\pi_h \tn_h^{m+\half}$ in the first equation of \eqref{CN-scheme-VCMCF} and follow the same procedures as for the surface diffusion, with  noting that 
    \begin{equation*}
        \big( H_h^{m+\half}-\langle H_h^{m+\half}\rangle,1\big)_{\hGam}=\big( H_h^{m+\half},1\big)_{\hat{\Gamma}_h^{m+\half}}-|\hGam| \langle H_h^{m+\half}\rangle=0,
    \end{equation*}
    the desired volume-preserving property follows. The energy-decaying property follows directly from the last equation of \eqref{CN-scheme-VCMCF}.
\end{proof}
\begin{remark}
    As an alternative, one can also apply the following Euler scheme to start the computation:
    \begin{equation}\label{Euler-scheme-VCMCF}
        \left\{
        \begin{aligned}
            \left(\frac{\vu_h^{1}-\vu_h^0}{\tau_{1}} ,\vxi_h\right)^h_{\Gamma_h^0}+     \left( H_h^{1}-\langle {H}_h^{1}\rangle,\cI_h\left[\frac{\vxi_h\cdot\pi_h\tn_h^{1}}{|\pi_h\tn_h^{1}|^2}\right]\right)_{\Gamma_h^0}&\\
            +\alpha\Big(\gMh \vX_h^{1},  \gMh\cI_h(\mathbb{P}_{h}^1\vxi_h\circ \vX_h^0)\Big)_{\cM_h}&=0, \quad \forall \vxi_h\in S(\Gamma_h^0)^3\\
            \left(\nabla_{\Gamma_h^0}\vu_h^{1}, \nabla_{\Gamma_h^0}(\eta_h\vnu_h^{0})\right)_{\Gamma_h^0}
            -\left(H_h^{1},\eta_h\right)^h_{\Gamma_h^0}&=0 ,\quad \forall \eta_h\in S(\Gamma_h^0)\\
            \frac{|\Gamma_h^{1}|-|\Gamma_h^0|}{\tau_1}+\left(\nabla_{\Gamma_h^0}H_h^{1},\nabla_{\Gamma_h^0}H_h^{1}\right)_{\Gamma_h^0}&=0,
        \end{aligned}
        \right.
    \end{equation}
    where $\tn_h^1$ and $\mathbb{P}_{h}^1$ are given in \eqref{Euler-n-P}. Numerical results indicate that using the Euler scheme \eqref{Euler-scheme-VCMCF} or the scheme mentioned above to initialize the method leads to no significant difference. 
\end{remark}

\subsection{Iterative method}
To solve the nonlinear system \eqref{CN-scheme-VCMCF}, we adopt an iterative method similar to \eqref{CN-scheme-linear}.
The initial guess is chosen as
\[(\vu_{h}^{m+1,0},H_{h,0}^{m+\half},\lambda_{h,0}^{m+1})=(\vu_{h}^{m},H_{h}^{m-\half},\lambda_{h}^{m})\]
and the solution is updated by successively solving the following linear system:
\begin{equation}\label{VCMCF-iteration}
    \left\{
    \begin{aligned}
        \left(\frac{\vu_{h,l+1}^{m+1}-\vu_h^m}{\tau_{m+1}} ,\vxi_h\right)^h_{\hGam}+\alpha\Big(\gMh \vX_{h,l+1}^{m+\half},
        \gMh\cI_h\big(\mathbb{P}_{h}^{m+\half}\vxi_h\circ \hX_h^{m+\half}\big)\Big)_{\cM_h}&\\
            + \lambda_{h,l+1}^{m+1}\left(
                H_{h,l}^{m+\half}-\langle H_{h,l}^{m+\half}\rangle,
                \cI_h\left[\frac{\vxi_h\cdot\pi_h\tn_{h,l}^{m+\half}}{|\pi_h\tn_{h,l}^{m+\half}|^2}\right]\right)_{\hGam}&\\
            +\lambda_{h,l}^{m+1}\left(H_{h,l+1}^{m+\half}- H_{h,l}^{m+\half},
                \cI_h\left[\frac{\vxi_h\cdot\pi_h\tn_{h,l}^{m+\half}}{|\pi_h\tn_{h,l}^{m+\half}|^2}\right]\right)_{\hGam}=0,&\\
        \left(\nabla_{\hGam}\frac{\vu_{h,l+1}^{m+1}+\vu_h^m}{2}, \nabla_{\hGam}\cI_h(\eta_h\vnu_{h,l}^{m+\half})\right)_{\hGam}
        =\left(H_{h,l+1}^{m+\half},\eta_h\right)^h_{\hGam},& \\
        \frac{|\Gamma_{h,l}^{m+1}|-|\Gamma_h^m|}{\tau_{m+1}}-\left(H_{h,l}^{m+\half}- \langle H_{h,l}^{m+\half} \rangle,H_{h,l}^{m+\half}-\langle H_{h,l}^{m+\half}\rangle \right)_{\hGam}&\\
        +\frac{1}{\tau_{m+1}} \int_{\Gamma_{h,l}^{m+1}}\nabla_{\Gamma_{h,l}^{m+1}}\cdot\big(\vu_{h,l+1}^{m+1}-\vu_{h,l}^{m+1}\big)\,\d_{\Gamma_h}&\\
        +2\left(H_{h,l+1}^{m+\half}- \langle H_{h,l+1}^{m+\half} \rangle, H_{h,l}^{m+\half}- \langle H_{h,l}^{m+\half} \rangle\right)_{\hGam}=0,&
    \end{aligned}
    \right.
\end{equation}
where 
\begin{equation*}
    \langle H_{h,l}^{m+\half} \rangle=\frac{1}{|\hGam|}\int_{\hGam}H_{h,l}^{m+\half} \,\d_{\Gamma_h}.
\end{equation*}
 For $m=0$, we apply a similar procedure where the integral domain $\hat{\Gamma}_h^{\half}$ is replaced by the image of $\vX_{h,l}^{\half}$ at the $(l+1)$th iteration with the initial domain $\Gamma_h^0$.  We stop the iteration and move to the next step if  $\max |\vu_{l+1}^{m+1}-\vu_{h,l}^{m+1}|\le tol$.

\section{Numerical examples}\label{Se-NE}
In this section, we present several numerical experiments for surface diffusion and the  volume-preserving mean curvature flow to demonstrate the volume-preserving and energy-decaying properties of our method, as well as the advantage of mesh distribution. Throughout this section, $N_p$ and $N_T$ denote the numbers of vertices and plane triangles in the mesh, respectively. Moreover, let  $U(t_m)$ denote the numerical solution at time $t=t_m$, with the $j$th component $U_j(t_m)\in \mathbb{R}^{3}$ representing the position of the $j$th vertex. The mesh quality of a two-dimensional polyhedral surface $\cN_h$ is measured by (see e.g., \cite[eq. (7.3)]{Elliott-Fritz-2017})
    \begin{equation}\label{msh_q}
        \sigma_{\rm max}(\cN_h)=\max_{T\in \cN_h}\frac{\mbox{diam } B^*(T)}{\mbox{diam } B_*(T)},
    \end{equation}
    where $B^*(T)$ and $B_*(T)$ denote the  circumcircle and the maximal inscribed circle of the triangle $T$, respectively. 
    
In addition, \Cref{prop-SDF} and \Cref{prop-MCF} indict that the Lagrange multiplier $\lambda$ can be arbitrary when $\Gamma(t)$ becomes a sphere. In practice, the iterative system will approach singular if the numerical surface is close to a sphere. This is also the case suffered in \cite{Garcke-Jiang-2025}. Similar to their strategy we remove $\lambda$ and the last equation in our schemes then solve the resulted systems. An empirical and sufficient criterion for such operation is when the energy decaying rate satisfies  $\frac{1}{\tau_{m}}\big({|\Gamma_h^{m}|-|\Gamma_h^{m-1}|}\big)\le 10^{-1}$. We utilized this rule in \Cref{ex-SDF-cuboid} and \Cref{ex-box-MCF}.

\subsection{Surface diffusion}
\begin{example}\label{ex-con_SDF}
    As the first example, we investigate the temporal convergence rates to verify the second-order accuracy. The initial surface is a $2:1:1$ ellipsoid whose mesh is generated by scaling the $x$-axis of an triangulated unit sphere by a factor of $2$. This triangulated sphere is used as the reference surface $\cM_h$, and we simply set $\alpha=1$ in all computations for this example. See \Cref{fig-ellipsoid-2-1-1} (a) for the reference mesh and \Cref{fig-ellipsoid-2-1-1} (b) for the initial mesh. The iterative tolerance we set in this example is  $tol=5\times 10^{-12}$.
    
    To test the temporal convergence rates, we fix the spatial mesh with $(N_p,N_T)=(1320,2636)$ to eliminate the pollution from spatial discretization errors, and run the scheme using several uniform time step sizes. We define 
    \begin{equation}\label{def-Err}
        {\rm Error}_t(t;\tau)=\max_{j}\left|U_j^{\tau}(t)-U_j^{\tau/2}(t)\right|,
    \end{equation}
where the superscript for $U(t)$ indicates the time stepsize. The temporal convergence rates is then computed by
\begin{equation}\label{rate_t}
    \mbox{Convergence rate}=\frac{\ln\left({\rm Error}_t(t;\tau)/{\rm Error}_t(t;\tau/2)\right)}{\ln 2}.
\end{equation}
We list ${\rm Error_t}(0.5;\tau)$ for various $\tau$ in \Cref{con-t-ex1}, by using \eqref{CN-scheme-start} and \eqref{Euler-scheme} as the starting step, respectively. One can observe  a second-order convergence rate in time when \eqref{CN-scheme-start} is applied at the first step, while the rate is reduced to about $1.5$ if \eqref{Euler-scheme} is chosen. Next we fix $\tau=1/200$ and examine the convergence of $\lambda_h$ with respect to the spatial mesh size at $t=0.3$ (slightly far away from being a sphere). The sequence of spatial mesh used is
\begin{equation}\label{sq-mesh}
    (N_p,N_T)=(1806,3608),\quad (3816,7628),\quad (7446,14888), \quad (11856,23708).
\end{equation}
Since the exact value of $\lambda$ is $1$ in the continuous model, we define the error of $\lambda$  by  ${\rm{Error}}_\lambda(h)=|\lambda_h-1|$ and compute the convergence rate by
\begin{equation}\label{rate-lam}
    \mbox{Convergence rate}=\frac{\ln\left({\rm Error}_\lambda(h_0)/{\rm Error}_\lambda(h'_0)\right)}{\ln (h_0/h'_0)}
\end{equation}
where $h_0$ and $h'_0$ are the maximal element diameters of two successive initial meshes. From \Cref{con-h-lam-SDF} we observe that $\lambda_h$ converges at an $\mO(h_0^2)$ rate as $h_0\rightarrow 0^+$. 

\Cref{fig-ellipsoid-2-1-1-EVI} displays the surface area, relative volume loss with respect to time, and the number of iterations at each step under different step sizes. As expected, the surface area decreases monotonically, while the enclosed volume is preserved up to a tolerance of $\mO(10^{-14})$. One also can observe the iterative number does not increase significantly when we reduce the step size. It costs more iterations at the first step since \eqref{CN-scheme-linear} is fully nonlinear.

\begin{table}[htp!]
	\centering
	\caption{\Cref{ex-con_SDF}: {${\rm Error}_t(0.5;\tau)$} under different $\tau$, $\tau_0=6.25\times 10^{-3}$.}\label{CT-SDF-t}
	\begin{tabular}{cccccc}
		\toprule
		$\tau$            & $\tau_0$        & $\tau_0/2$     & $\tau_0/4$   & $\tau_0/8$ & $\tau_0/16$\\ 
		\toprule
		Error (using \eqref{CN-scheme-start})    
		&  6.251e-04  & 1.013e-04 &  7.601e-06  & 1.828e-06   & 4.531e-07 \\ 
		Convergence rate  & --     &2.63&   3.74& 2.06  & 2.01  \\ 
		\bottomrule
		Error (using \eqref{Euler-scheme})
		&   1.306e-04  & 5.001e-05 &  1.884e-05  & 6.919e-06  & 2.497e-06  \\ 
		Convergence rate  & --     &1.39 &    1.41  &    1.45  &    1.47  \\ 
		\bottomrule
	\end{tabular}\label{con-t-ex1}
\end{table}    
\begin{table}[htp!]
    \centering
    \caption{\Cref{ex-con_SDF}: {${\rm Error}_\lambda(h)$} under different spatial meshes at $t=0.3$.}
    \label{con-h-lam-SDF}
    \begin{tabular}{ccccc}
        \toprule

        $h_0$             & 2.35e-01 &  1.44e-01 &  1.05e-01 &  8.55e-02 \\ 
        \toprule

        Error & 1.318e-02  & 6.478e-03 &  3.368e-03  & 2.102e-03 \\
        Convergence rate &     -- &   1.44 &   2.05 &   2.36 \\
        \bottomrule

    \end{tabular}
\end{table}    
\begin{figure}[!htp]
    \centering
    \subfigure[Reference surface]{\includegraphics[width=.26\textwidth,trim=0 140 0 170, clip]{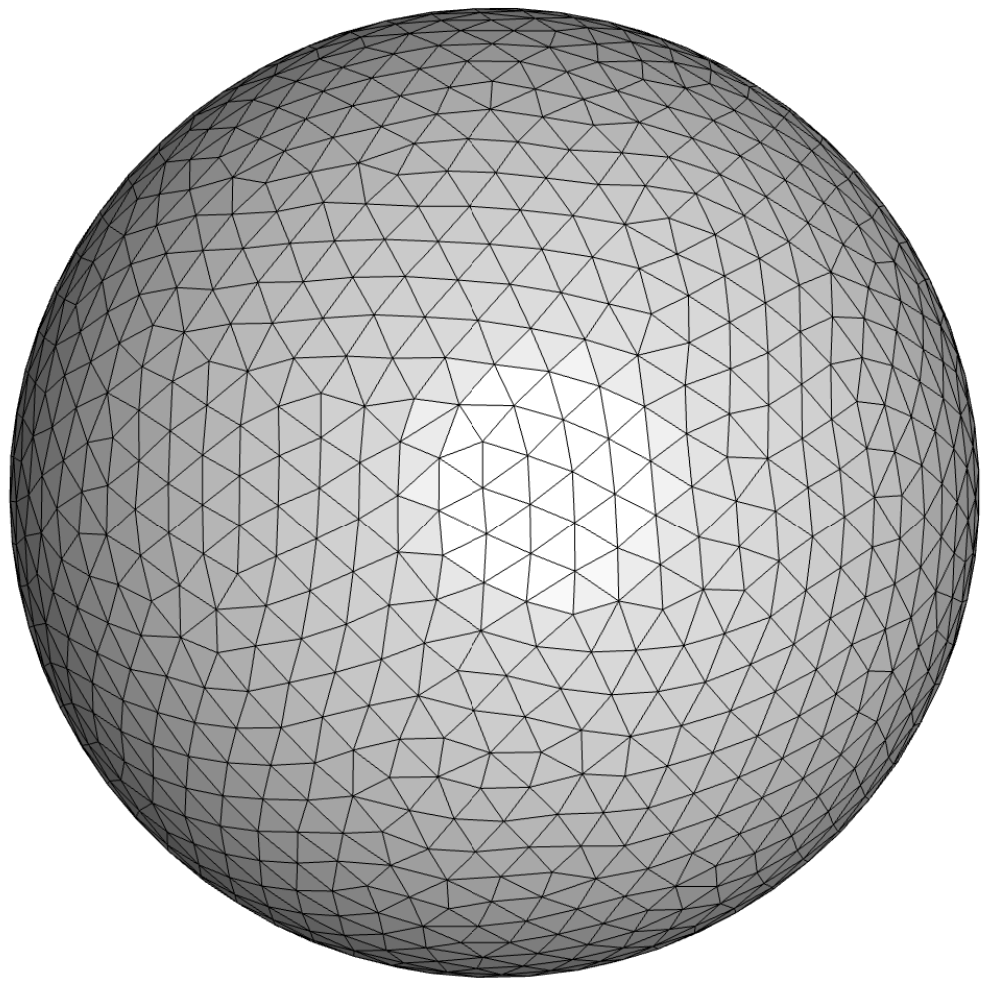}}
    \subfigure[Initial surface]{\includegraphics[width=.36\textwidth,trim=0 220 0 180, clip]{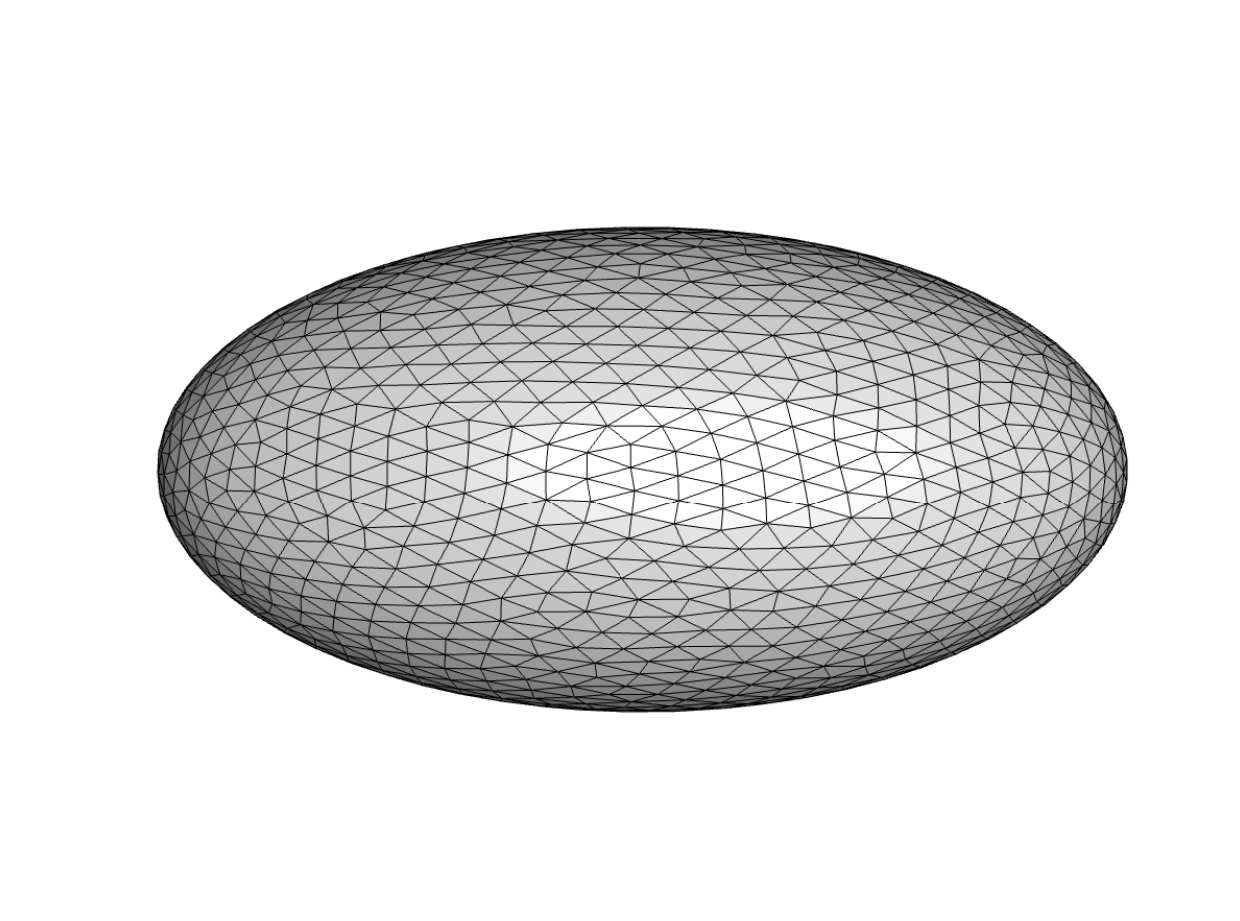}}
    \subfigure[$t=0.5$]{\includegraphics[width=.26\textwidth,trim=0 140 0 170, clip]{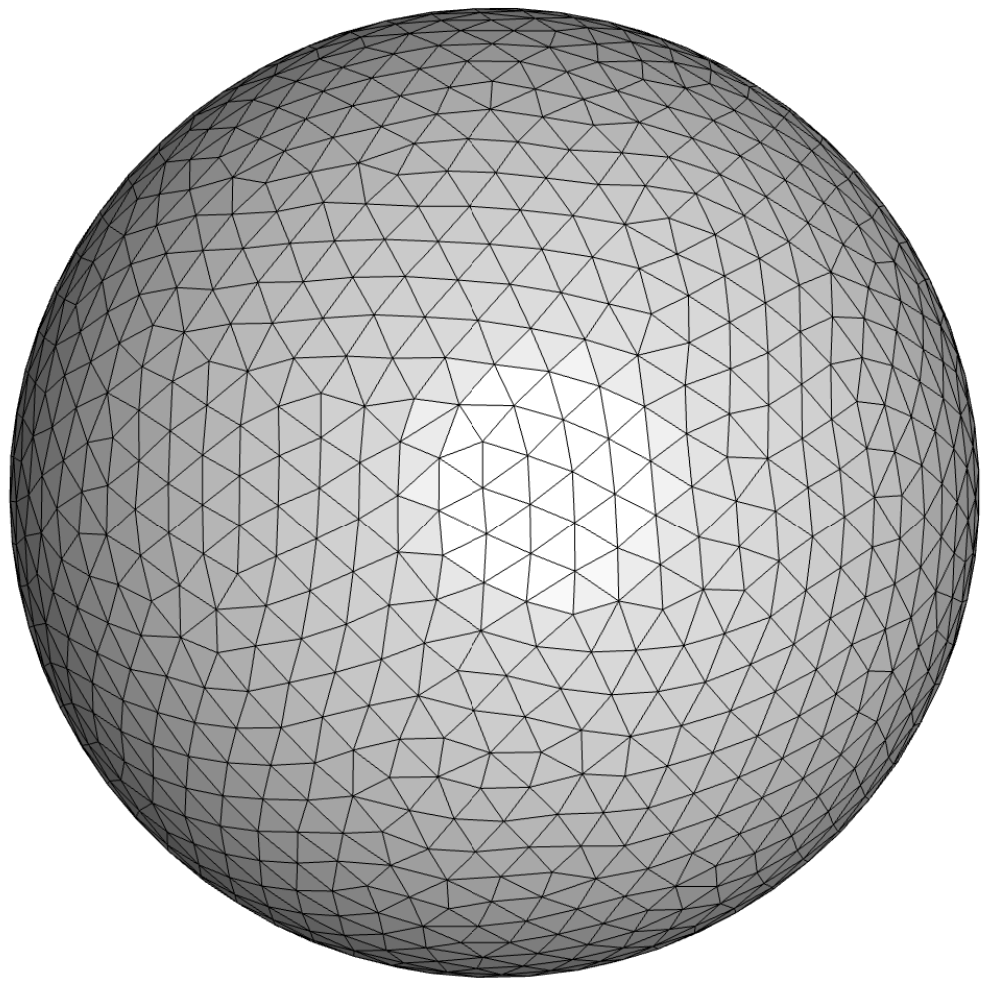}}
    \caption{Reference surface (a), initial surface (b),  and numerical surface at $t=0.5$ (c) for \Cref{ex-con_SDF} with $(N_p,N_T)=(1320,2636)$, the images are rescaled.}
    \label{fig-ellipsoid-2-1-1}
\end{figure}

\begin{figure}[!htp]
    \centering
    \subfigure[Surface area]{\includegraphics[width=.32\textwidth,trim=0 200 10 220, clip]{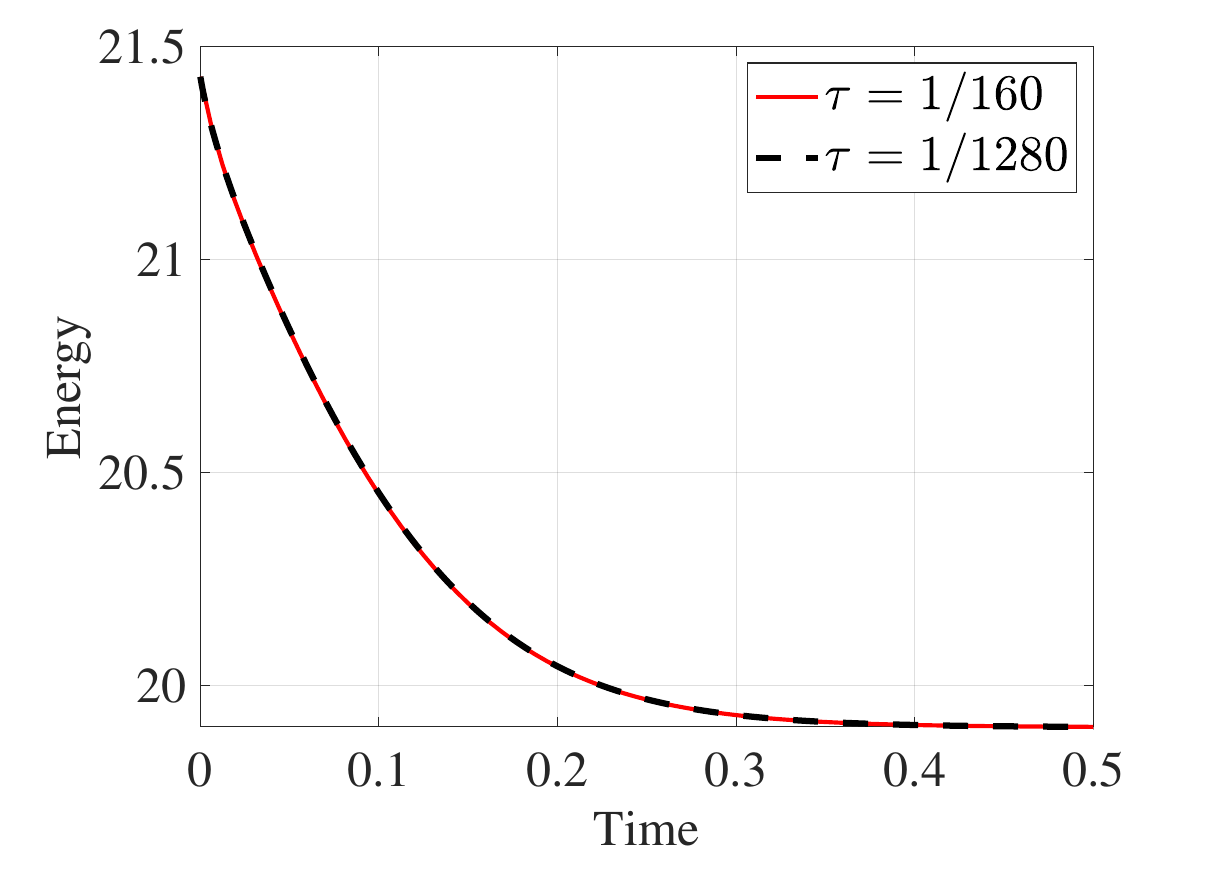}}
    \subfigure[Volume]{\includegraphics[width=.32\textwidth,trim=0 200 10 210, clip]{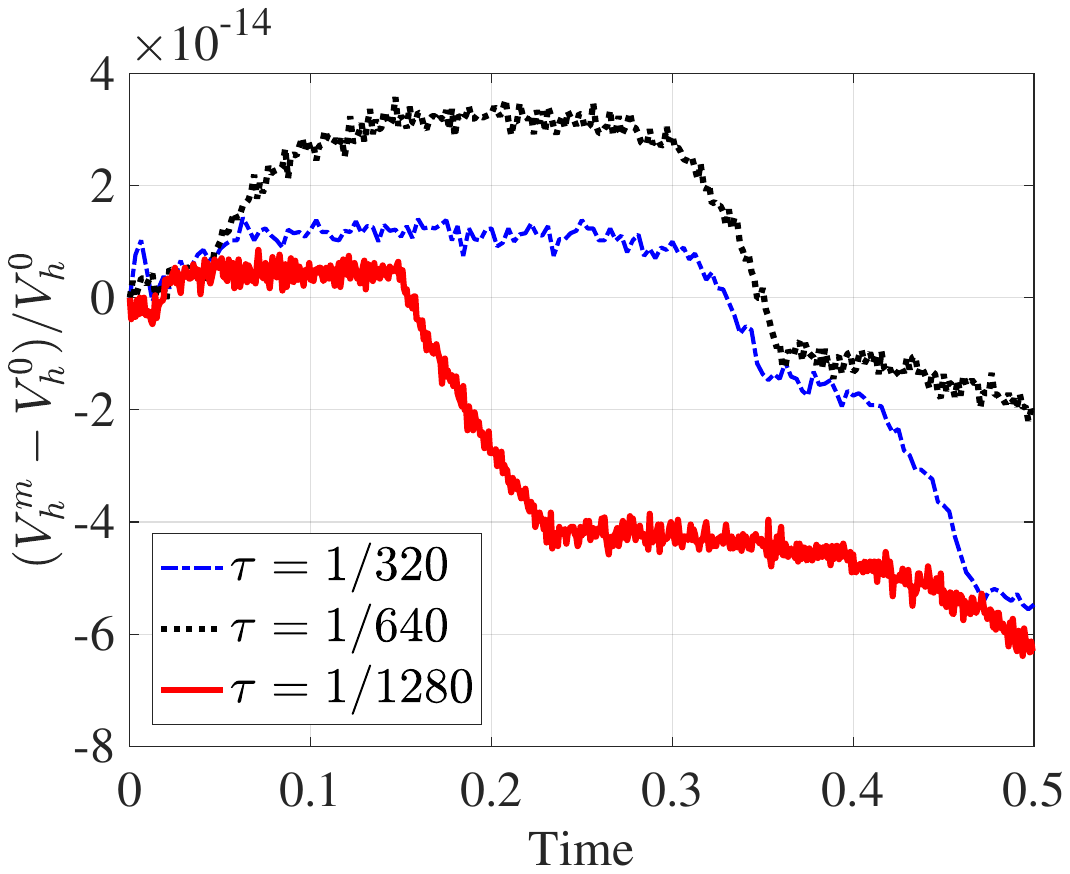}}
    \subfigure[Iterations]{\includegraphics[width=.32\textwidth,trim=0 200 10 220, clip]{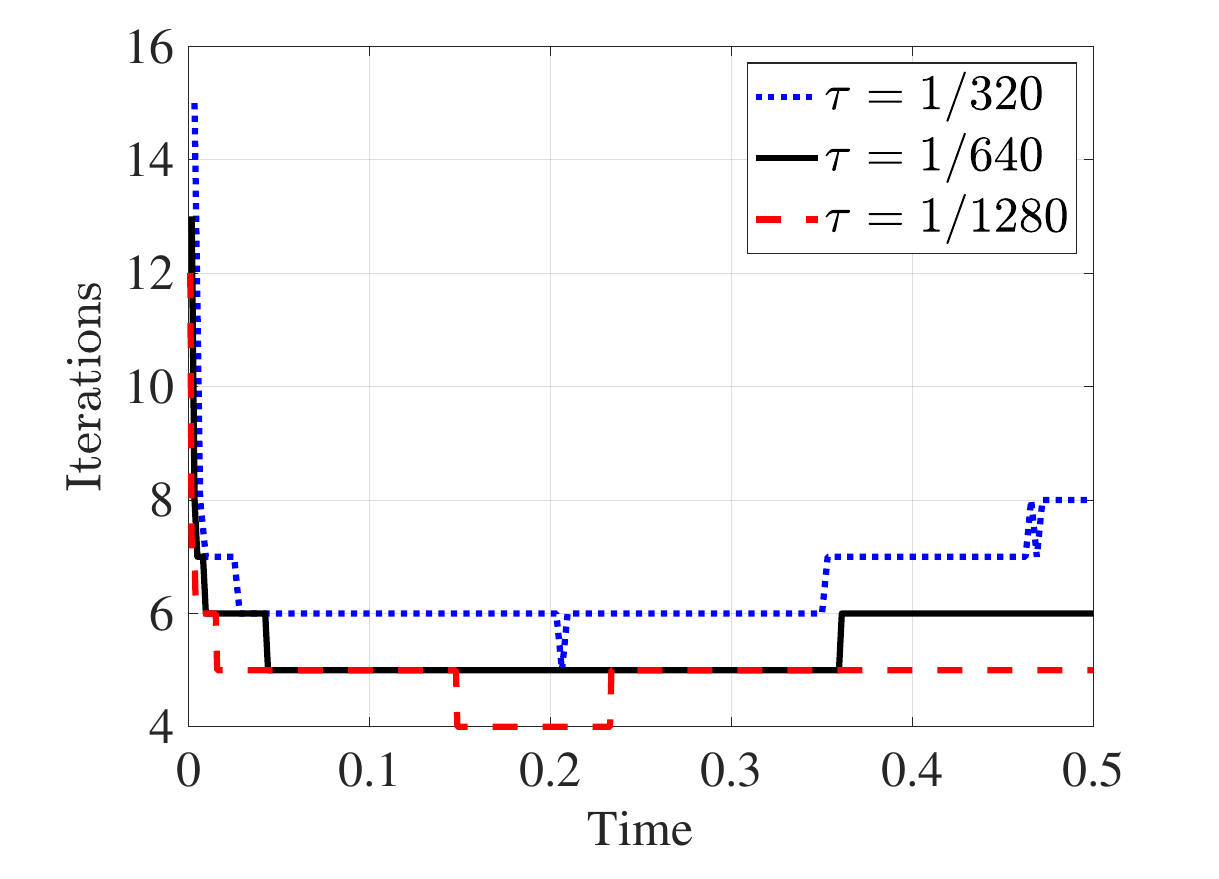}}
    \caption{Energy with respect to time (a), relative volume loss (b),  and iterations at each step (c) for \Cref{ex-con_SDF} with $(N_p,N_T)=(1320,2636)$ and $tol=5\times 10^{-12}$.}
    \label{fig-ellipsoid-2-1-1-EVI}
\end{figure}
\end{example}

\begin{example}\label{ex-SDF-H}
    In this example, we consider the surface diffusion of an H-shaped surface whose initial mesh is shown in \Cref{fig-H-initial} (a) with $(N_p,N_T)=(3153,6302)$. We compare our scheme with the SP-Euler and SP-CN schemes proposed in \cite{Garcke-Jiang-2025}. Since the initial surface has sharp corners, we run all the schemes under graded mesh in time whose nodes are given by $t_m = T(\frac{m}{M})^2,\, m=0,1\cdots, M$. We set $T=1$, $M=2000$ and $tol=1\times 10^{-10}$ in this example.

    Numerical results of different schemes for $t=0.04$, $t=0.16$ and $t=1$ are presented in \Cref{fig-H-SDF} where we set $\alpha=10$ in our scheme. One can see that meshes produced by our scheme are better than those of SP-Euler and SP-CN schemes at $t=1$. For the SP-CN scheme, instead of extrapolation, they applied the original BGN scheme to predict $\hGam$. So we present the number of solvers that different schemes cost in \Cref{fig-H-SDF} for comparison. One can see that except the first step where we used a fully nonlinear scheme, the cost of our method is at the same level as theirs.
    
     \Cref{fig-H-cmp} reports the surface area, enclosed volume, and mesh quality versus time for the above mentioned three schemes. We used a log-scaled $x$-axis for time to plot the energy since it decreases very fast at the beginning. One can see that all curves of energy match each other very well except the beginning parts when the energy decreases very fast.  The relative volume loss of our method is $\mO(10^{-12})$ while theirs is $\mO(10^{-15})$. This is because in their method, the enclosed volume is forced to be constant by an extra Lagrange multiplier which means only machine errors involved. \Cref{fig-H-cmp} (c) presents the mesh quality index $\sigma_{{\rm max}}$ as a function of time. For our method, the mesh of the concave part undergoes compression in one direction and stretching in another direction, so $\sigma_{{\rm max}}$ increases at the beginning. Due to the continuous effort of the harmonic map heat flow, $\sigma_{{\rm max}}$ reduces to the same level as $\sigma_{{\rm max}}(0)$. Compared with our method, the meshes produced by SP-Euler and SP-CN schemes keep getting worse. In fact, $\sigma_{\rm max}(1)\approx 772$ for the SP-CN scheme.

\begin{figure}[!htp]
    \centering
    \subfigure[Initial mesh]{\includegraphics[width=0.5\textwidth, trim=450 245 380 210, clip]{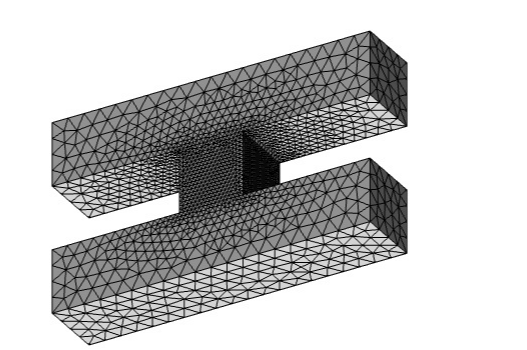}}

    \subfigure[Number of Solvers]{\includegraphics[width=0.35\textwidth]{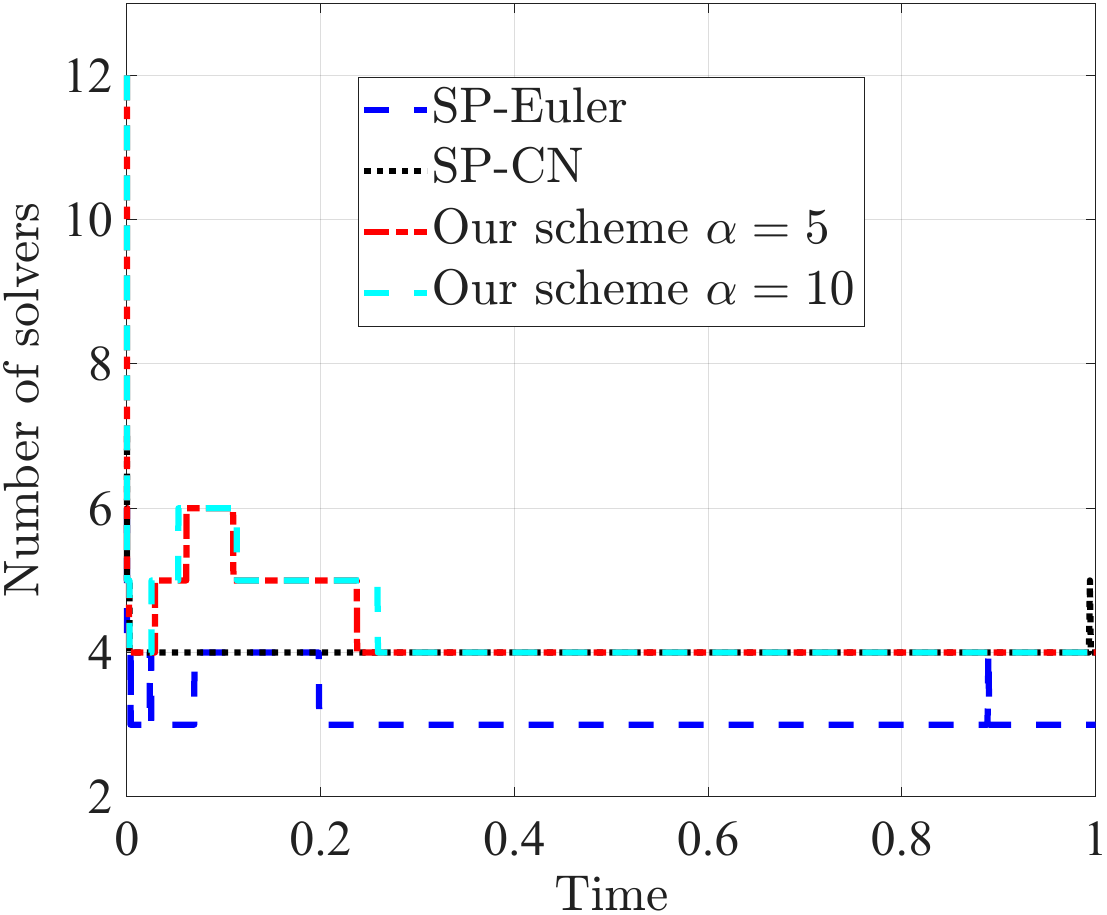}}
    \caption{Initial mesh and number of solvers of different schemes for \Cref{ex-SDF-H}.}
    \label{fig-H-initial}
\end{figure}
\begin{figure}[!htp]
    \centering
    \subfigure[SP-Euler:   $t=0.04$]{\includegraphics[width=0.32\textwidth, trim=250 200 250 180, clip]{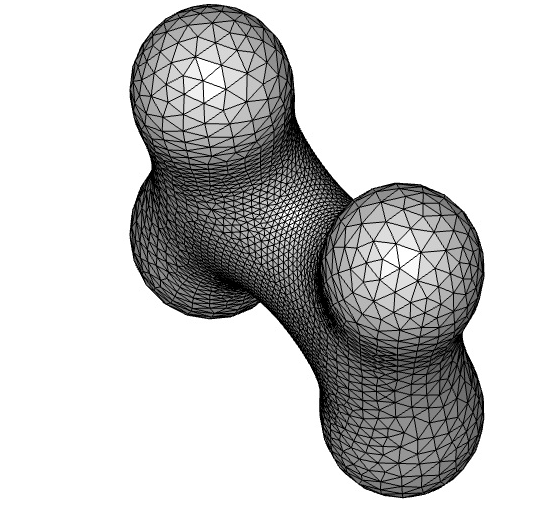}}
    \subfigure[SP-Euler:   $t=0.16$]{\includegraphics[width=0.26\textwidth, trim=250 180 250 150, clip]{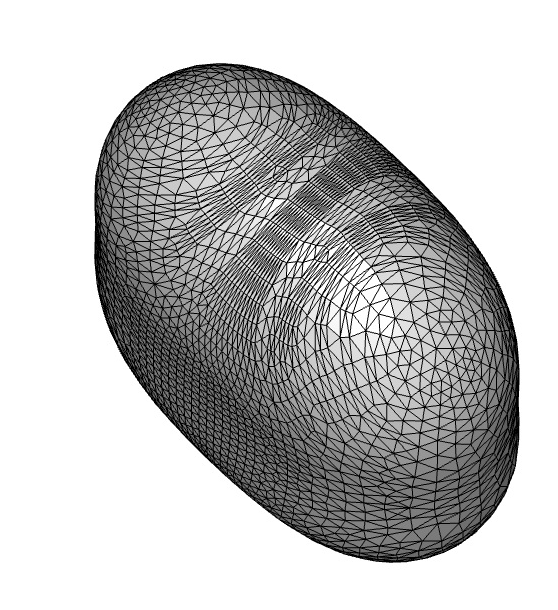}}
    \subfigure[SP-Euler:      $t=1$]{\includegraphics[width=0.26\textwidth, trim=250 160 200 150, clip]{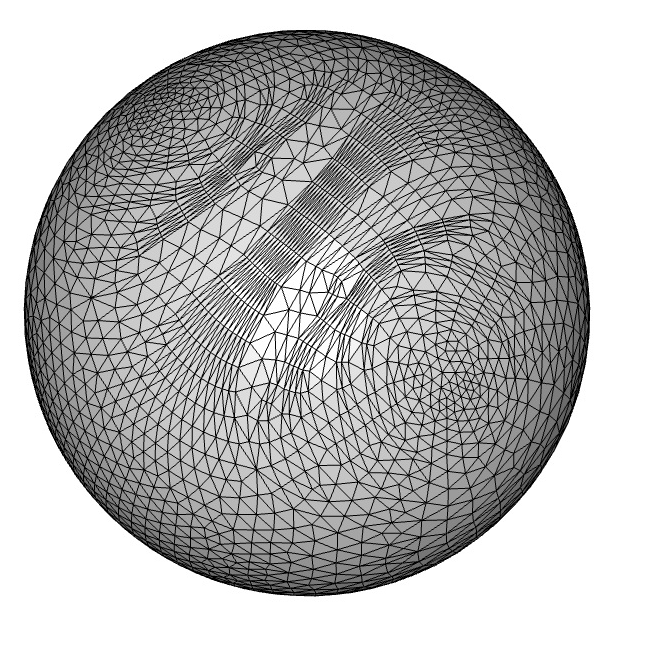}} \\
    \subfigure[SP-CN:      $t=0.04$]{\includegraphics[width=0.32\textwidth, trim=250 200 250 180, clip]{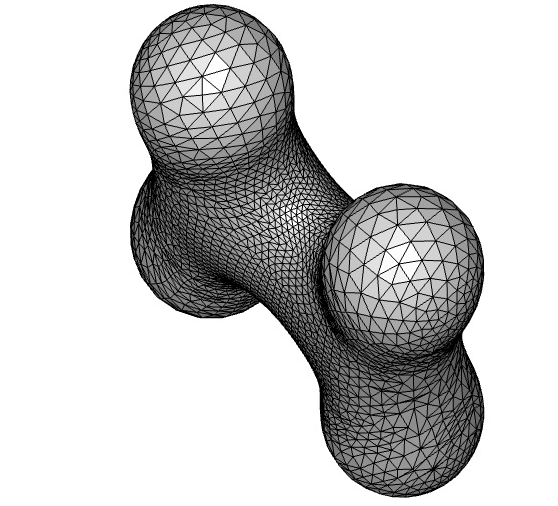}}              
    \subfigure[SP-CN:      $t=0.16$]{\includegraphics[width=0.26\textwidth, trim=250 180 250 150, clip]{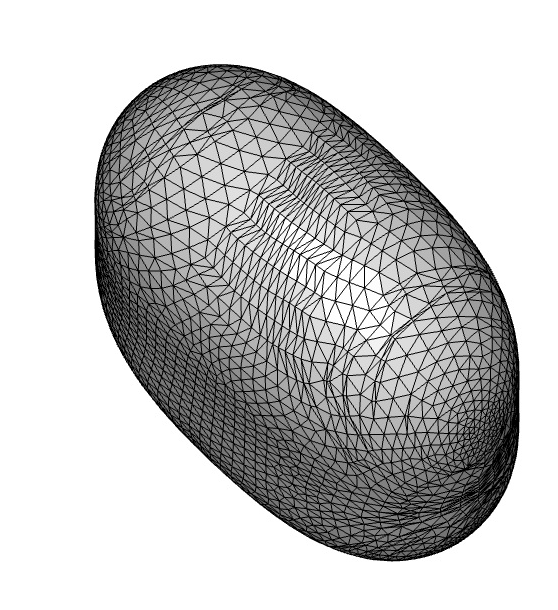}}              
    \subfigure[SP-CN:         $t=1$]{\includegraphics[width=0.26\textwidth, trim=250 160 200 150, clip]{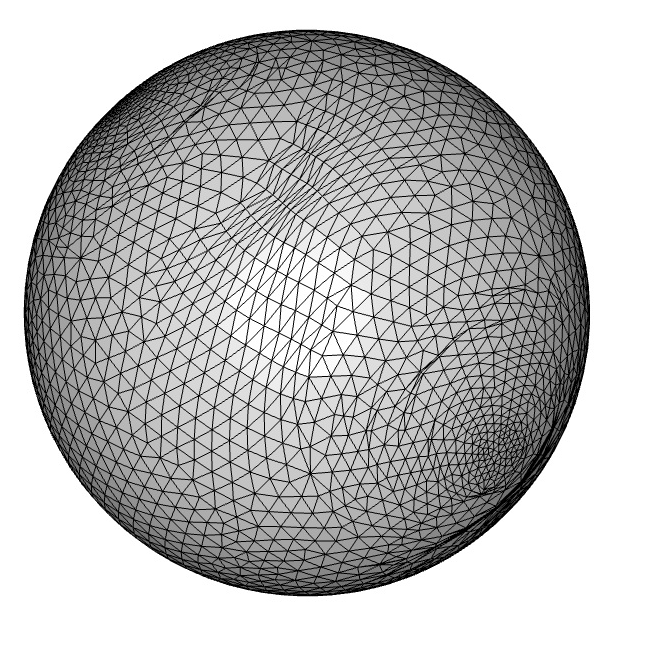}}        \\      
    \subfigure[Our scheme: $t=0.04$]{\includegraphics[width=0.32\textwidth, trim=250 200 250 180, clip]{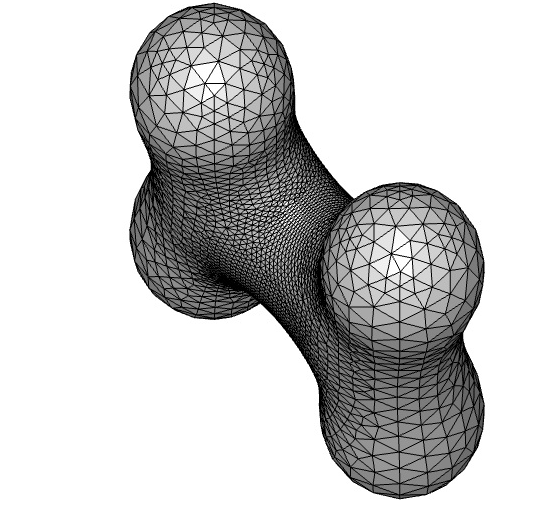}}              
    \subfigure[Our scheme: $t=0.16$]{\includegraphics[width=0.26\textwidth, trim=250 180 250 150, clip]{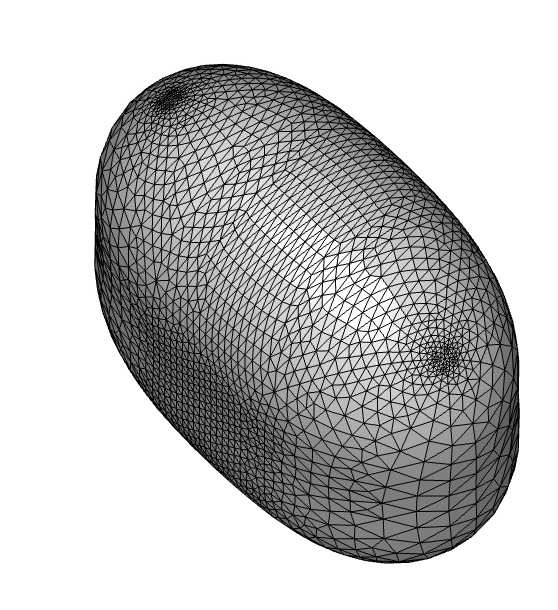}}                 
    \subfigure[Our scheme:    $t=1$]{\includegraphics[width=0.26\textwidth, trim=250 160 200 150, clip]{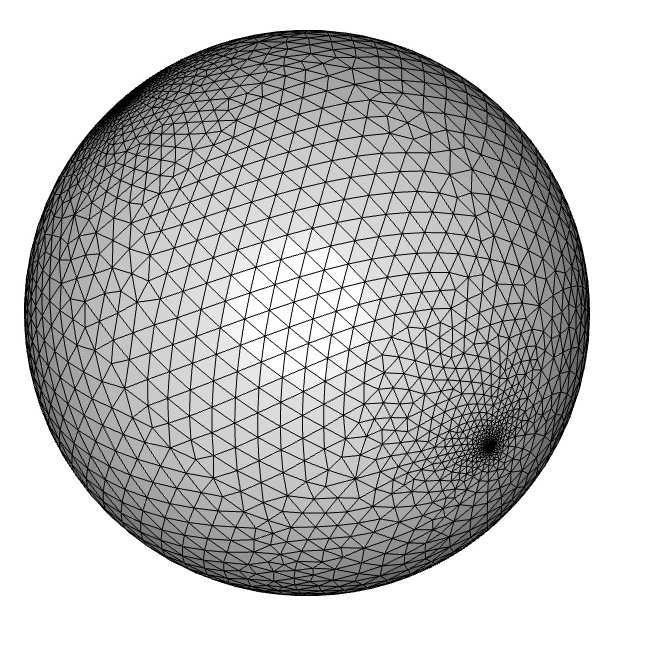}} \\

    \caption{Numerical results of difference schemes for \Cref{ex-SDF-H}.}
    \label{fig-H-SDF}
\end{figure}

\begin{figure}[!htp]
    \centering
    \subfigure[Surface area                      ]{\includegraphics[width=0.3\textwidth]{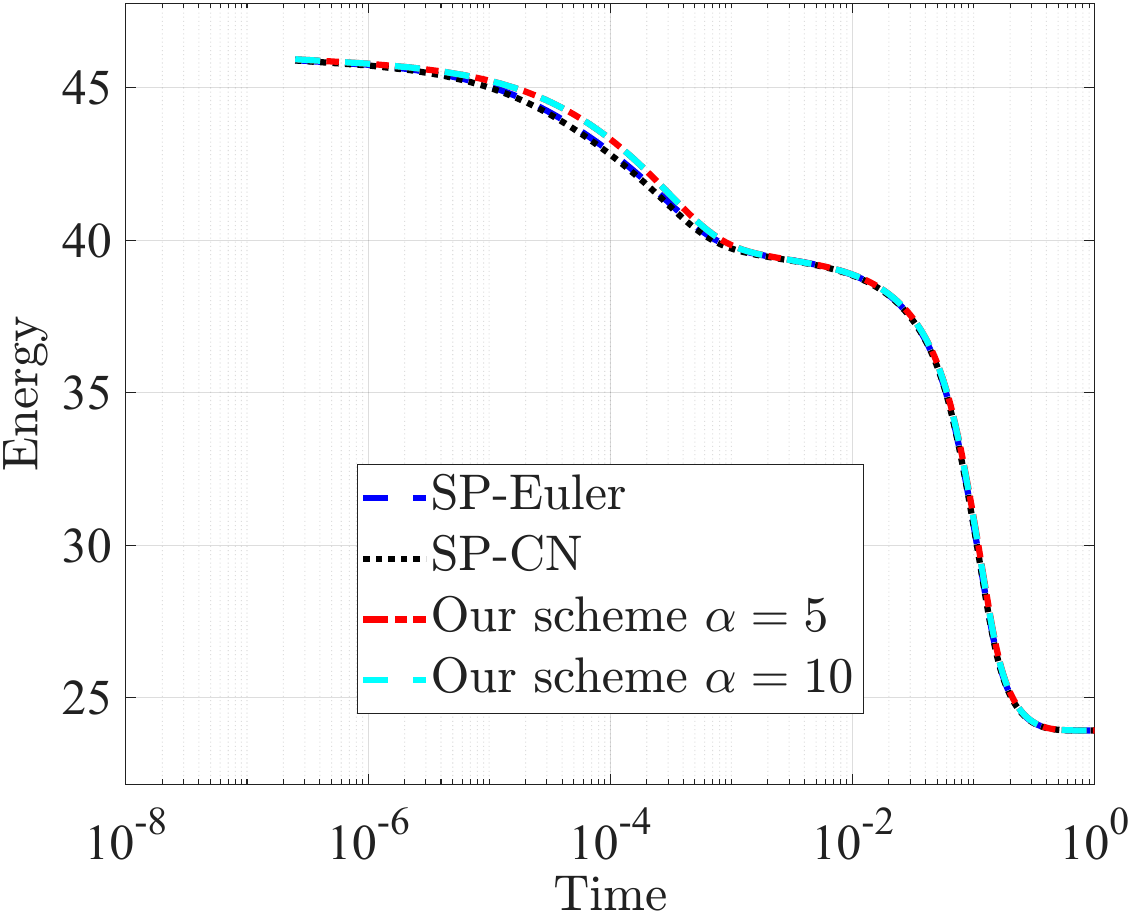}}
    \subfigure[Volume                 ]{\includegraphics[width=0.3\textwidth]{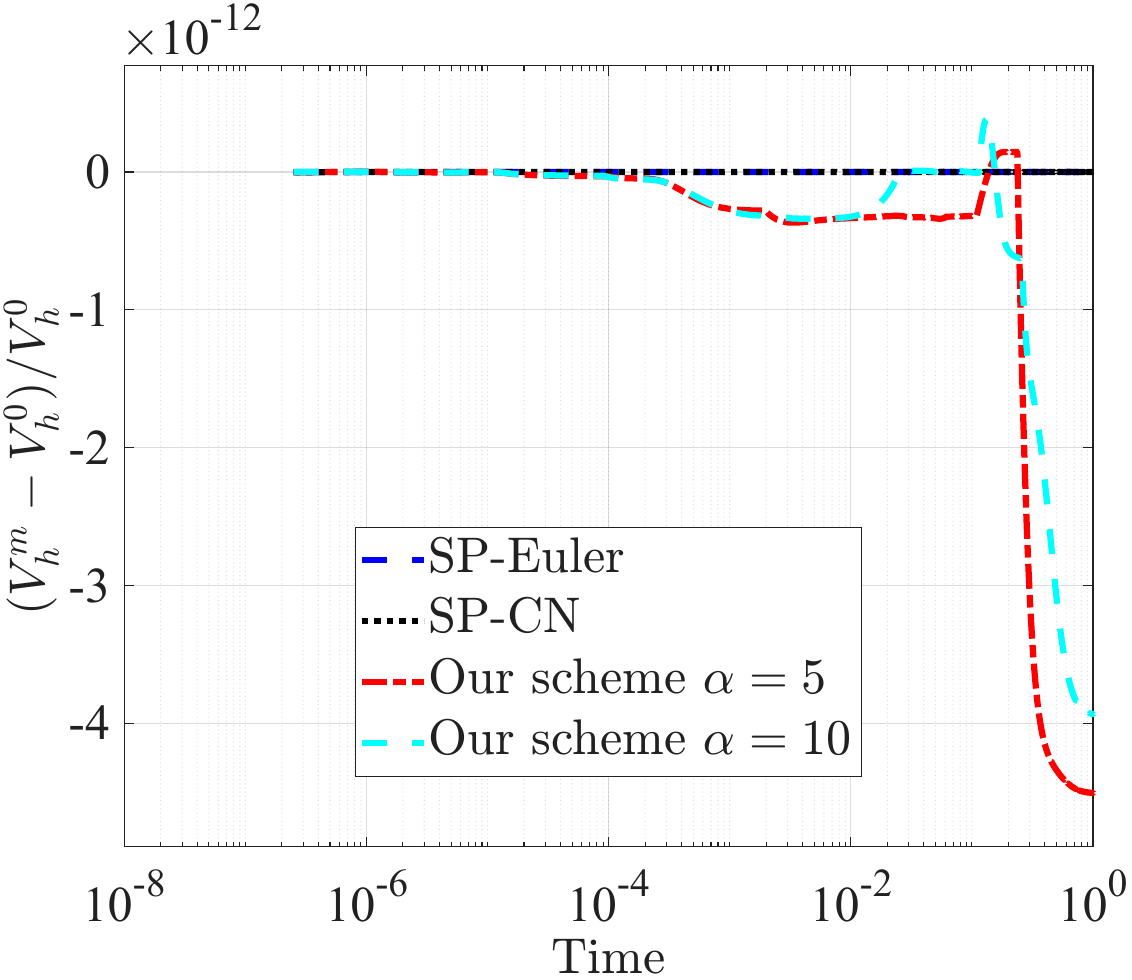}}
    \subfigure[Mesh quality]{\includegraphics[width=0.3\textwidth]{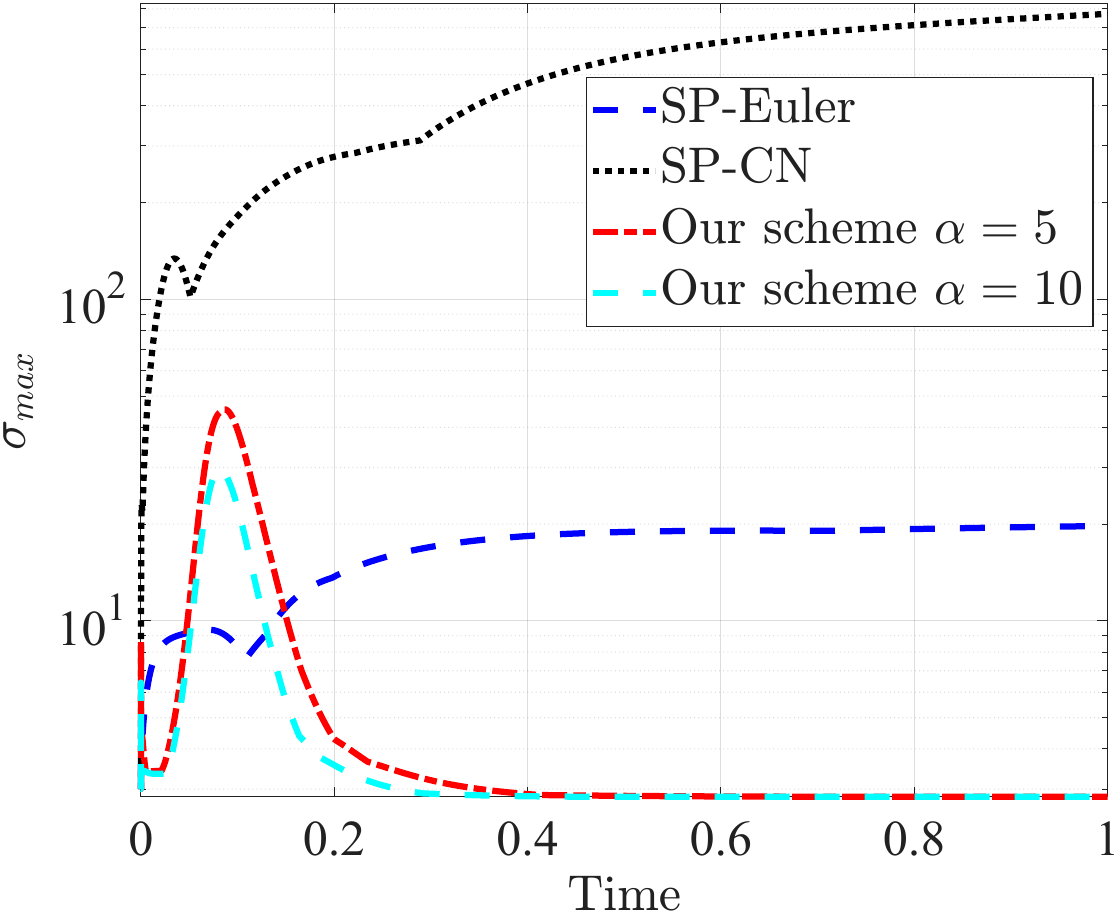}}

    \caption{Comparison of different schemes for \Cref{ex-SDF-H}.}
    \label{fig-H-cmp}
\end{figure}
\end{example}
\begin{example}\label{ex-SDF-cuboid}
    We consider the surface diffusion of a $1:1:6$ box  to compare our scheme with the SP-Euler and SP-CN schemes proposed in \cite{Garcke-Jiang-2025}. The initial mesh is shown in \Cref{fig-cuboid-initial} with $(N_p,N_T)=(804,1604)$. The time nodes are chosen as $t_m = T(\frac{m}{M})^2, m=0,1,\cdots, M$ with $T=1$ and $M=500$. The iterative tolerance is $tol=5\times 10^{-12}$.

    Numerical results of different schemes are shown in \Cref{fig-cuboid-SDF}, indicating the advantage of our schemes on mesh distribution.  The surface area, enclosed volume, and mesh quality with respect to time are plotted in \Cref{fig-cuboid-cmp} (a)-(c), respectively.  \Cref{fig-cuboid-cmp} (b) presents the relative volume loss which shows ours is $\mO(10^{-13})$ and theirs is around the machine error. This is because, as we have mentioned, we use the averaged-in-time normal vector  to implicitly maintain the volume which may get iteration error involved, while their schemes explicitly require $V_h^m=V_h^0$.  \Cref{fig-cuboid-cmp} (c) presents $\sigma_{{\rm max}}$ from which one can see that compared with their schemes our method can preserve the mesh as well as $\Gamma_h^0$.

\begin{figure}[!htp]
    \centering
    \includegraphics[width=0.4\textwidth, trim=180 290 180 240, clip]{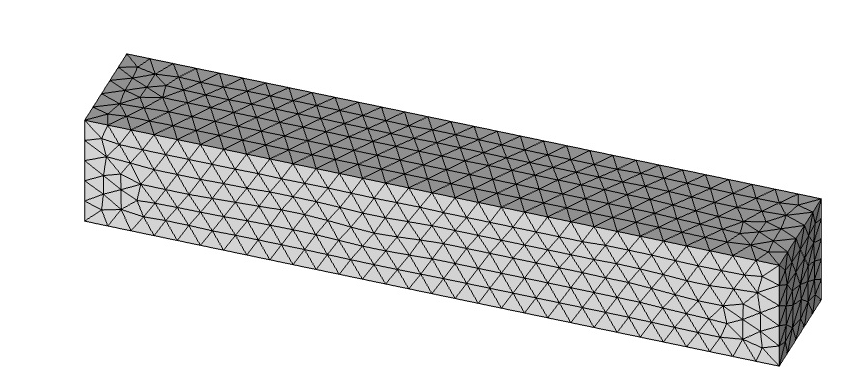}
    \caption{Initial mesh of cuboid shape for \Cref{ex-SDF-cuboid}.}\label{fig-cuboid-initial}
\end{figure}
\begin{figure}[!htp]
    \centering
    \centering
    \subfigure[SP-CN:   $t=0.0576$]{\includegraphics[width=0.38\textwidth, trim=200 260 200 260, clip]{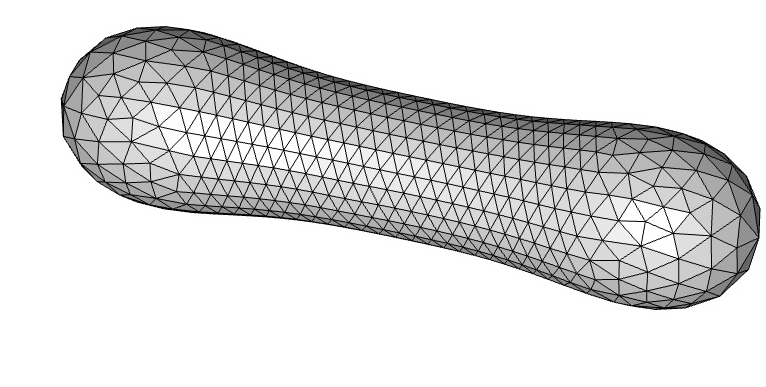}} 
    \subfigure[SP-CN:     $t=0.25$]{\includegraphics[width=0.34\textwidth, trim=200 260 200 260, clip]{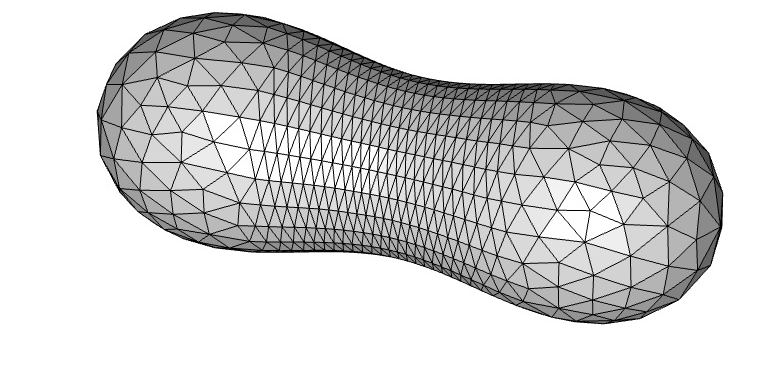}}              
    \subfigure[SP-CN:        $t=1$]{\includegraphics[width=0.26\textwidth, trim=160 160 160 150, clip]{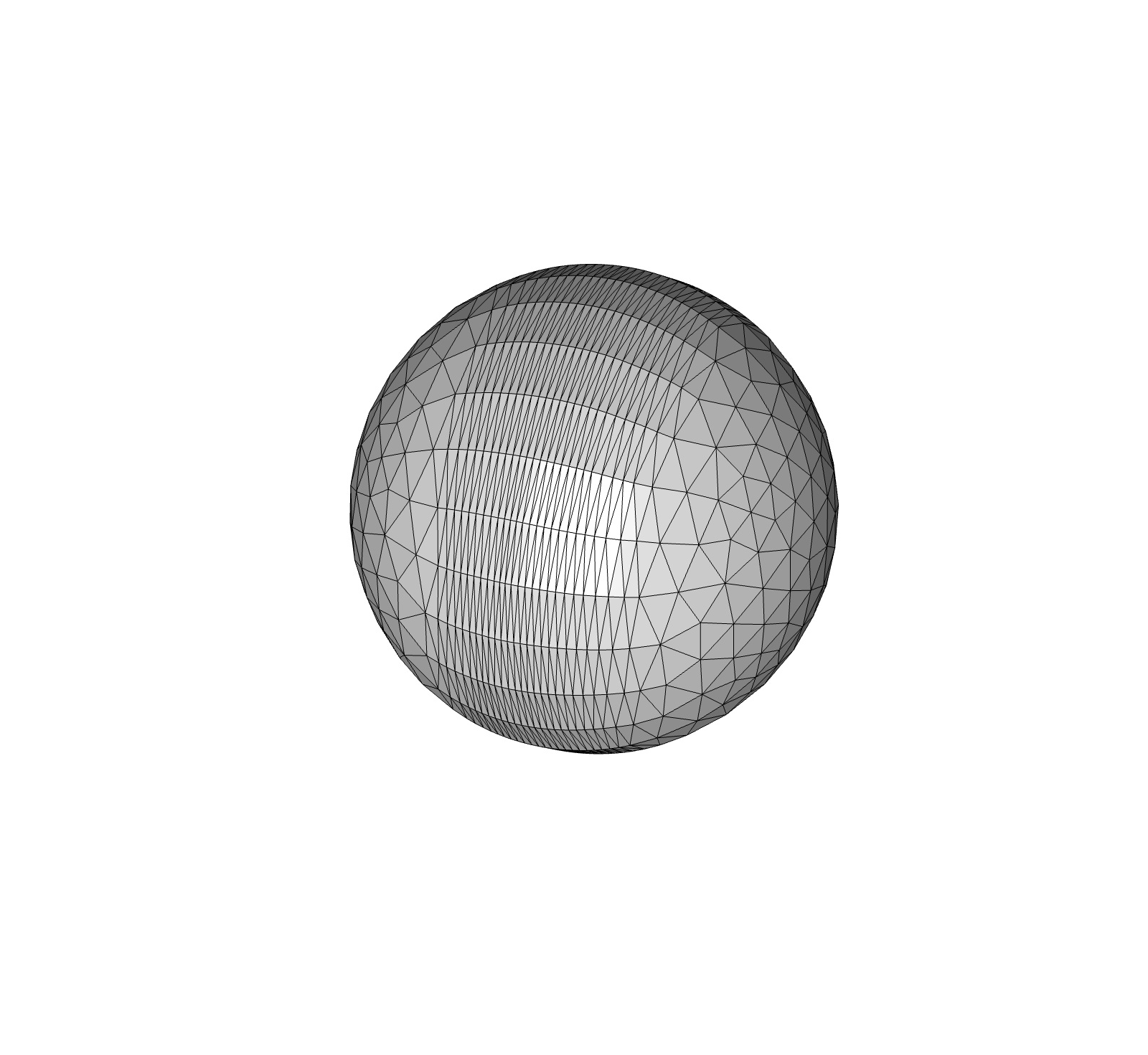}} \\
    \subfigure[Our scheme:   $t=0.0576$]{\includegraphics[width=0.38\textwidth, trim=200 250 200 250, clip]{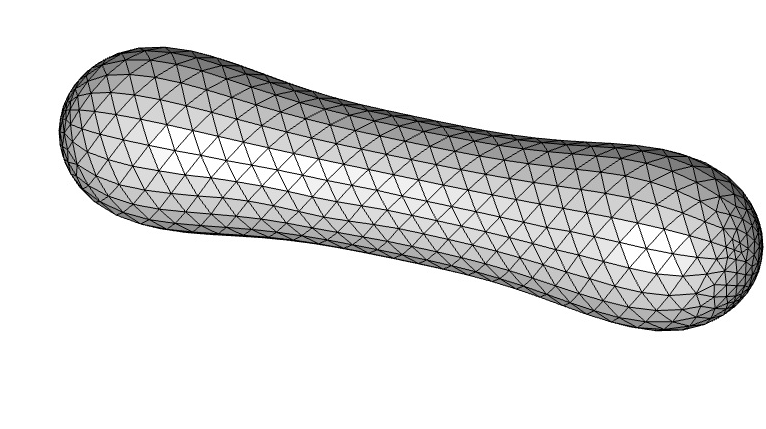}}
    \subfigure[Our scheme:     $t=0.25$]{\includegraphics[width=0.34\textwidth, trim=200 250 200 250, clip]{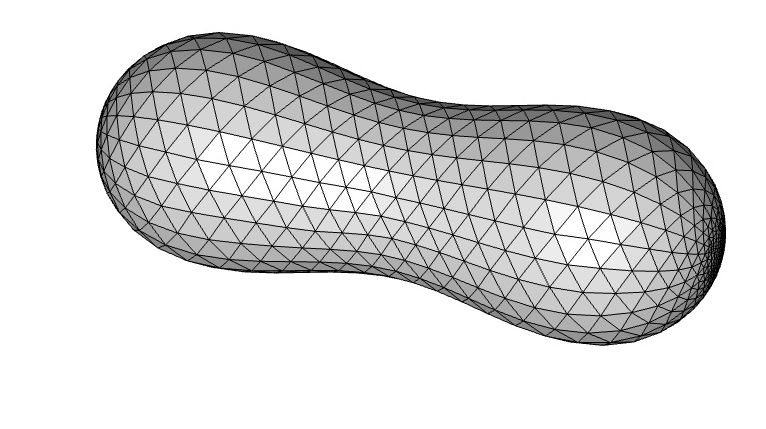}}              
    \subfigure[Our scheme:        $t=1$]{\includegraphics[width=0.26\textwidth, trim=160 160 160 100, clip]{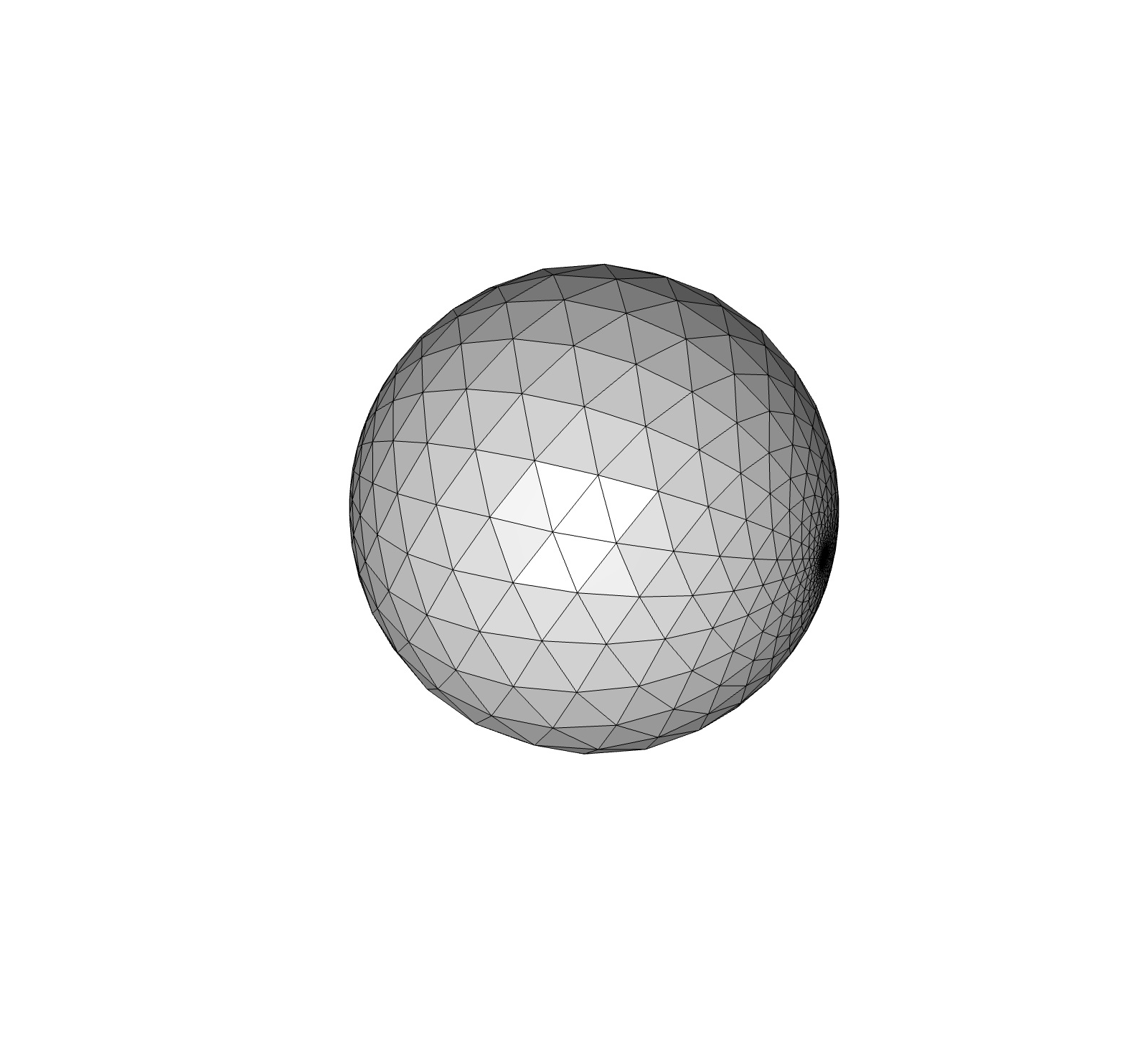}}
    \caption{Surface diffusion of a $1:1:6$ box using different schemes (SP-CN in the first row and ours with $\alpha=5$ in the second row) for \Cref{ex-SDF-cuboid}. }
    \label{fig-cuboid-SDF}
\end{figure}

\begin{figure}[!htp]
    \centering
\subfigure[Surface area        ]{\includegraphics[width=0.3\textwidth]{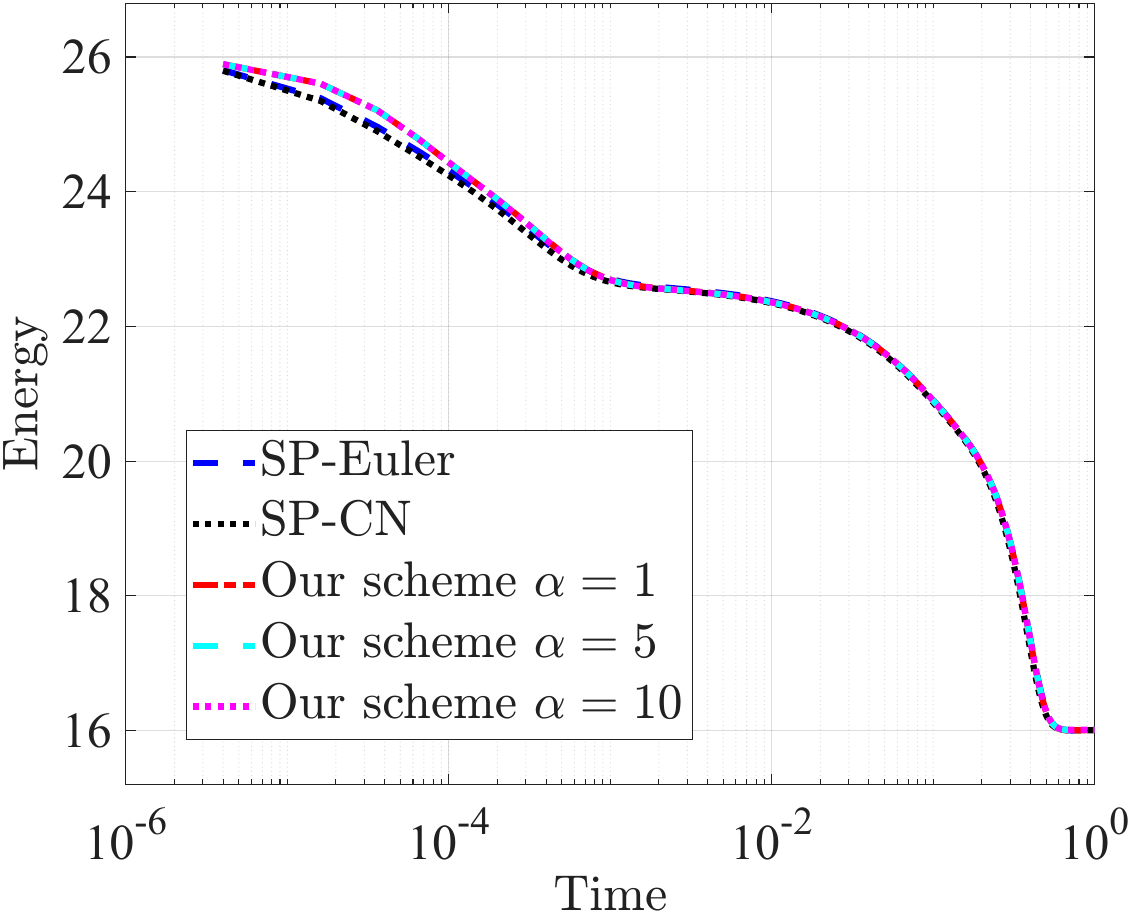}}
\subfigure[Volume                         ]{\includegraphics[width=0.3\textwidth]{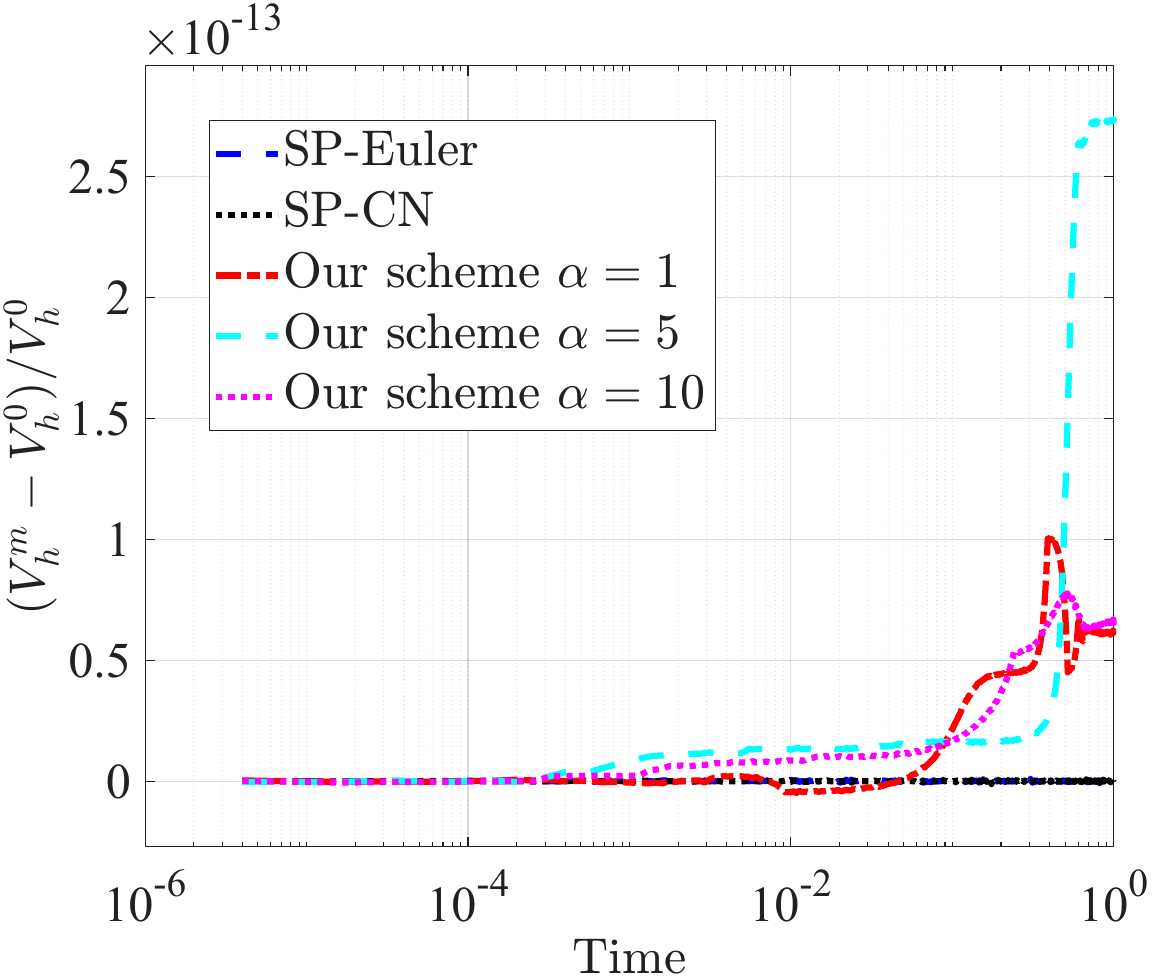}}
\subfigure[Mesh Quality        ]{\includegraphics[width=0.3\textwidth]{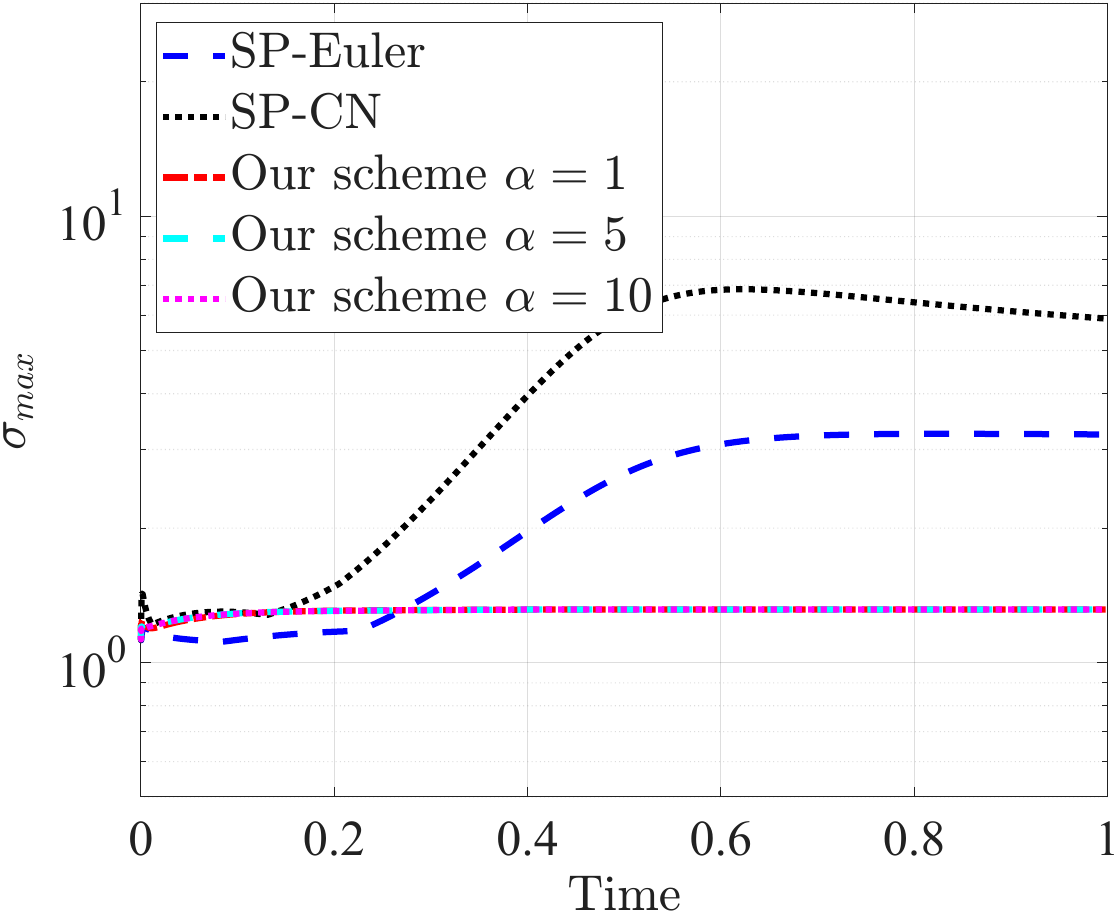}}
\caption{Comparison of different schemes on a $1:1:6$ box for \Cref{ex-SDF-cuboid}.}
\label{fig-cuboid-cmp}
\end{figure}
\end{example}

\subsection{Volume-preserving mean curvature flow}
\begin{example}\label{ex-con-MCF}
    We first examine the temporal convergence rates of scheme \eqref{CN-scheme-VCMCF}. The same reference and initial surfaces as in \Cref{ex-con_SDF} are used, and ${\rm Error_t}(t;\tau)$ is computed with $tol=5\times 10^{-12}$ for different time step sizes at $t=1$. The convergence rates are then calculated by \eqref{rate_t} and listed in \Cref{con-t-ex_MCF}, from which a second-order temporal accuracy is observed. Next, we adopt the same procedure as in the case of surface diffusion  to test the convergence rate of $\lambda_h$ with respect to the mesh size of $\Gamma_h^0$ at $t=1$. The mesh sequence is given in \eqref{sq-mesh}, and the step size is fixed as $\tau=2\times 10^{-3}$. \Cref{con-h-lam-MCF} indicates that the convergence rate of $\lambda_h$ is $\mO(h_0^2)$.

    \begin{table}[htp!]
	\centering
	\caption{\Cref{ex-con-MCF}: {${\rm Error}_t(1;\tau)$} under different $\tau$, $\tau_0=2\times 10^{-2}$.}\label{CT-MCF-t}
	\begin{tabular}{cccccc}
		\toprule
		$\tau$             & $\tau_0$        & $\tau_0/2$     & $\tau_0/4$   & $\tau_0/8$ & $\tau_0/16$\\ 
		\toprule
		Error                & 8.451e-05  & 2.131e-05 &  5.355e-06  & 1.342e-06  & 3.358e-07\\ 
		Convergence rate      & --     &1.99 &   2.00 & 2.00  & 2.00   \\ 
		\bottomrule
	\end{tabular}\label{con-t-ex_MCF}

\end{table}    
\begin{table}[htp!]
\centering
\caption{\Cref{ex-con-MCF}: {${\rm Error}_\lambda(h)$} under different spatial meshes at $t=1$.}
\begin{tabular}{ccccc}
	\toprule
	$h_0$             &  2.35e-01  &   1.44e-01  &   1.05e-01  &  8.55e-02 \\ 
	\toprule
	Error    &5.330e-03   & 2.676e-03  &  1.362e-03  & 8.545e-04 \\ 
	Convergence rate  & --     &1.40 &   2.15  & 2.33    \\ 
	\bottomrule
\end{tabular}\label{con-h-lam-MCF}
\end{table}    
\end{example}

\begin{example}\label{ex-Torus-MCF}
    We consider the evolution of a flower type torus defined by the parameterization
    \begin{equation}\label{eq-torus-65}
        \vx=\left[
        \begin{matrix}
            (1+0.4\cos\varphi)\cos\theta \\
            (1+0.4\cos\varphi)\sin\theta \\
            0.4\sin\varphi+0.2\sin6\theta
        \end{matrix}
        \right],
        \quad \theta\in [0,2\pi) ,\quad\varphi \in [0,2\pi) . 
    \end{equation}
We take $\Gamma(0)=\{\vx:=[x_1,x_2,x_3]\}$ as the image of the following mapping defined on a $(1,0.4)$-torus $\cM$
\begin{equation}\label{torus-flower}
    x_1=z_1,\quad x_2=z_2,\quad  x_3=z_3+0.2\sin\left(6\arctan\frac{z_2}{z_1}\right),\qquad [z_1,z_2,z_3]\in \cM.
\end{equation}
The initial triangulation $\Gamma_h^0$ is constructed by mapping the triangulated $(1,0.4)$-torus through \eqref{torus-flower}. See \Cref{fig-Torus-MCF} (a) for the reference surface $\cM_h$ and \Cref{fig-Torus-MCF} (b) for the initial surface $\Gamma_h^0$, where 
 $(N_p, N_T)=(4217,8434)$. \Cref{fig-Torus-MCF} (c)-(f) display the numerical surfaces at $t=0.02$, $t=0.05$, $t=0.1$ and $t=0.2$, respectively. The initially undulating surface gradually flattens and eventually develops a singularity at the center. The energy-decaying and volume-preserving properties are verified in \Cref{fig-Torus-MCf-EVI} (a) and (b), respectively. \Cref{fig-Torus-MCf-EVI} (c) presents the mesh quality index $\sigma_{\rm max}$ with respect to time, indicating that good mesh quality is maintained for evolution far from singularities. In fact, $\sigma_{{\rm max}}(0.2)\approx \sigma_{{\rm max}}(0)=19.58$ for $\alpha=5$ or $10$. 
\begin{figure}[!htp]
    \centering
    \subfigure[Reference surface]{\includegraphics[width=.31\textwidth,trim=0 200 0 190, clip]{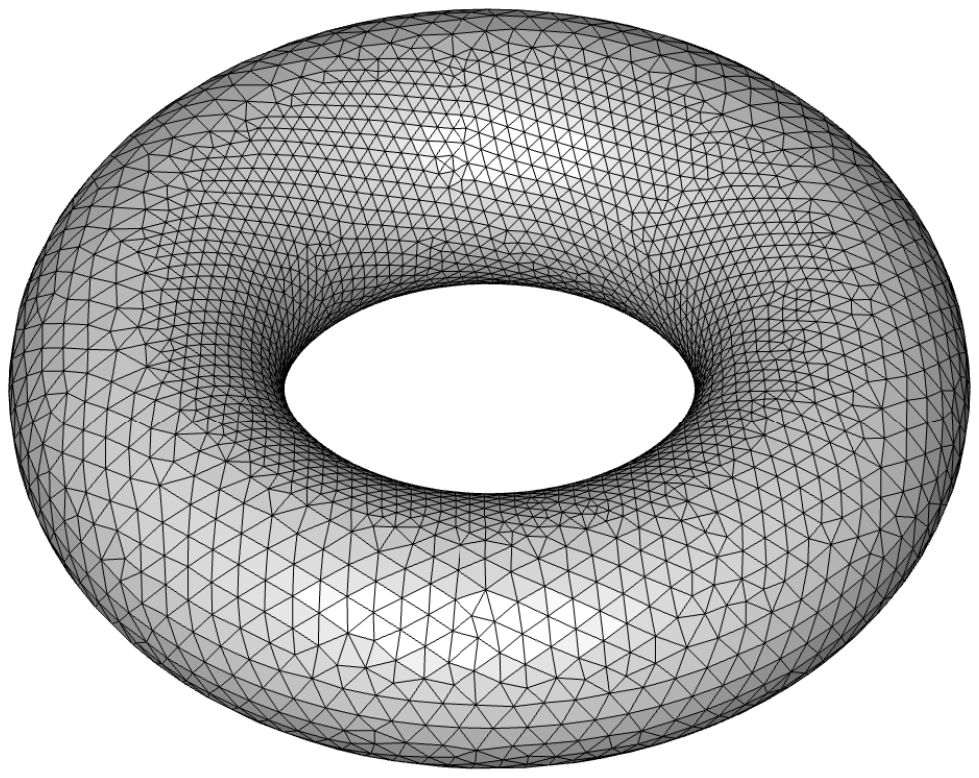}}
    \subfigure[Initial surface]{\includegraphics[width=.31\textwidth,trim=0 200 0 190, clip]{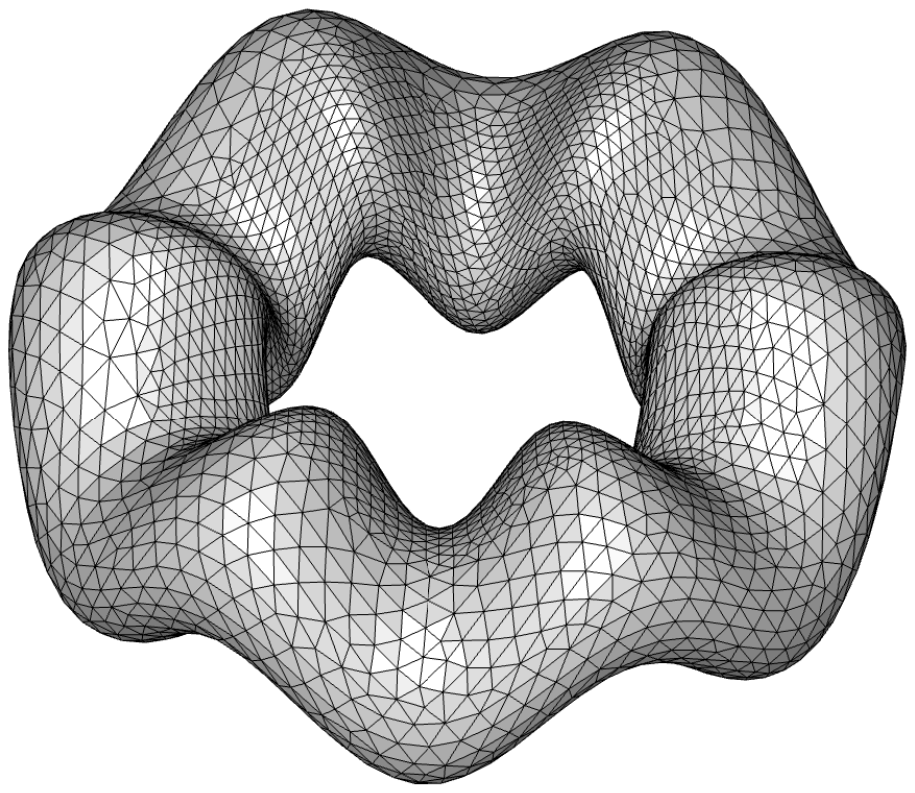}}
    \subfigure[$t=0.02$]{\includegraphics[width=.31\textwidth,trim=0 200 0 190, clip]{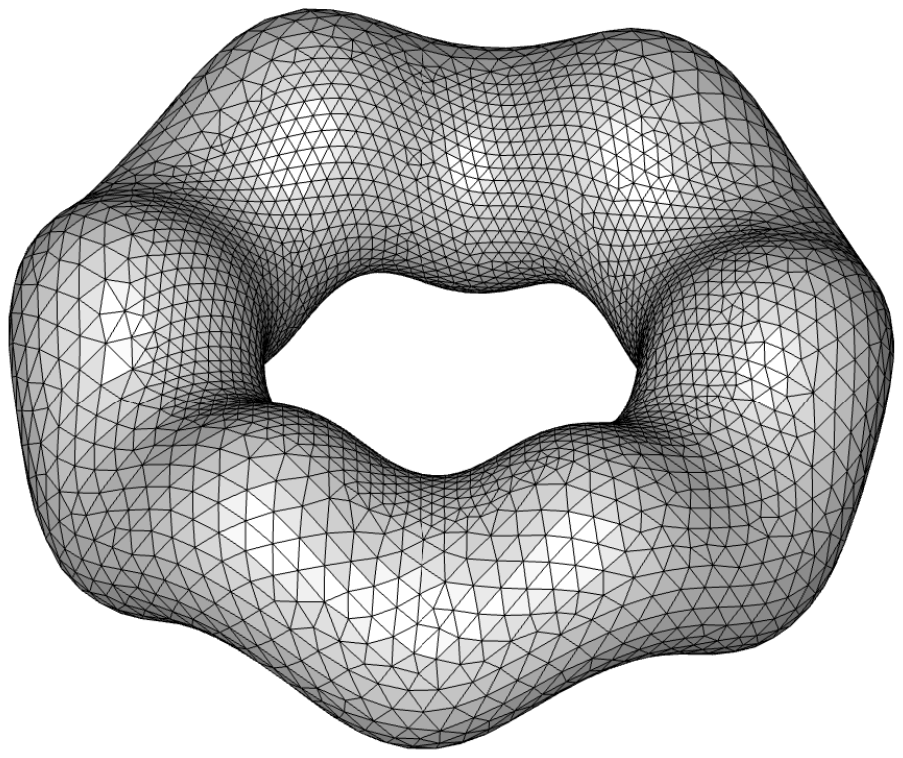}}\\
    \subfigure[$t=0.05$]{\includegraphics[width=.31\textwidth,trim=0 200 0 190, clip]{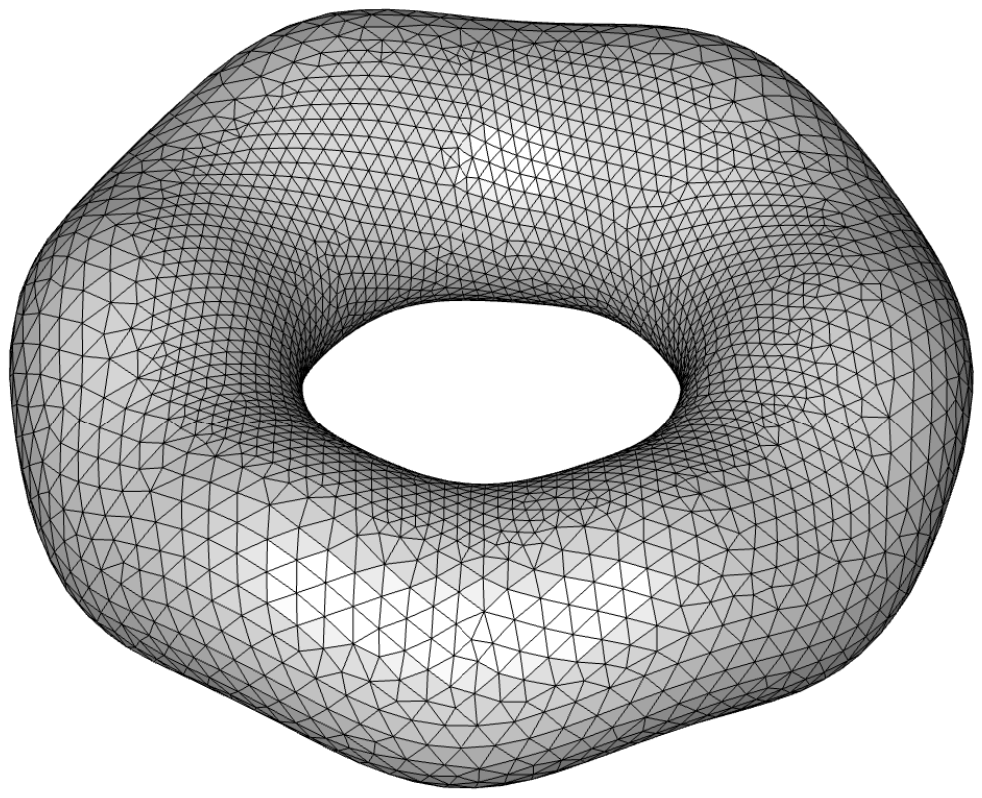}}
    \subfigure[$t=0.1$]{\includegraphics[width=.31\textwidth,trim=0 200 0 190, clip]{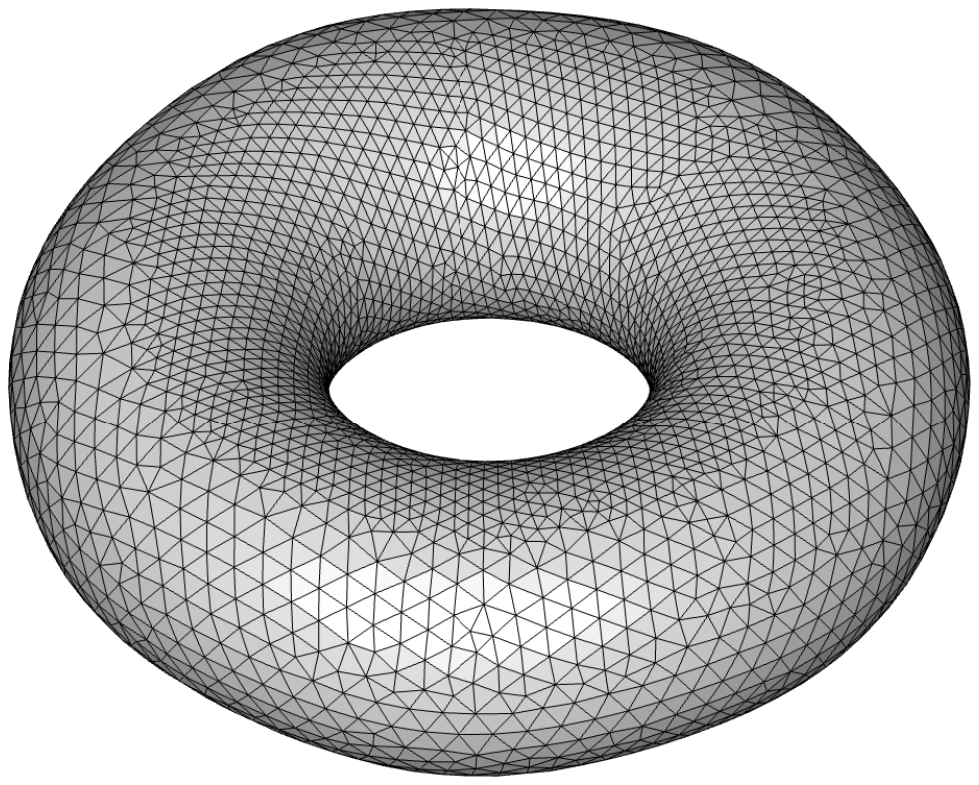}}
    \subfigure[$t=0.2$]{\includegraphics[width=.31\textwidth,trim=0 200 0 190, clip]{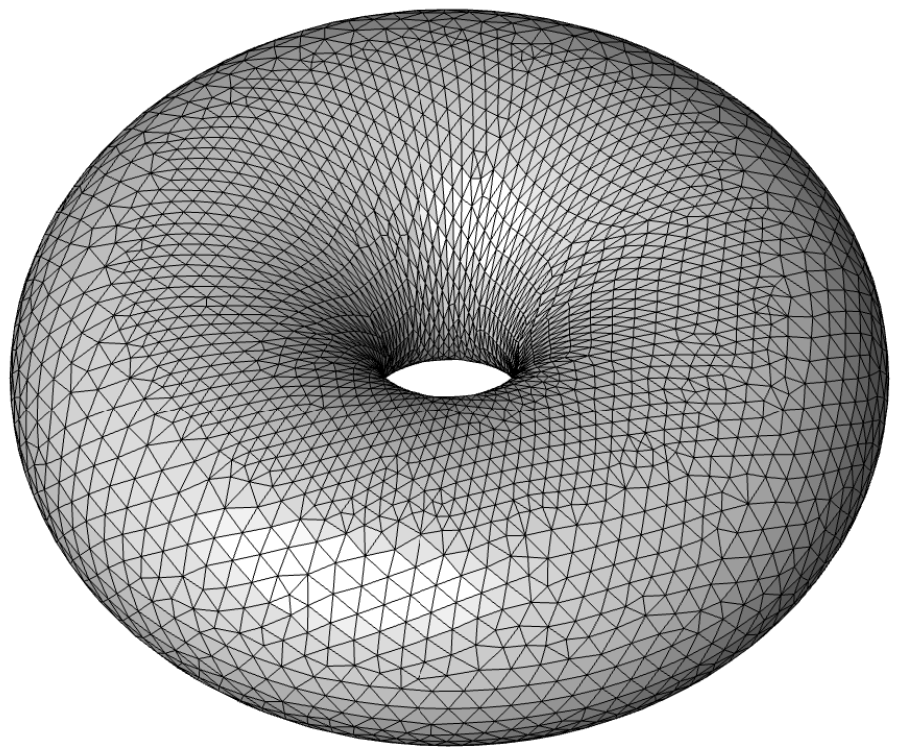}}
    \caption{Reference surface (a), initial surface (b). Plots (c)-(f) are numerical surfaces at $t=0.02$, $t=0.05$, $t=0.1$ and $t=0.2$, respectively.}
    \label{fig-Torus-MCF}
\end{figure}
\begin{figure}[!htp]
    \centering
    \subfigure[Surface area]{\includegraphics[width=.32\textwidth,trim=0 200 0 200, clip]{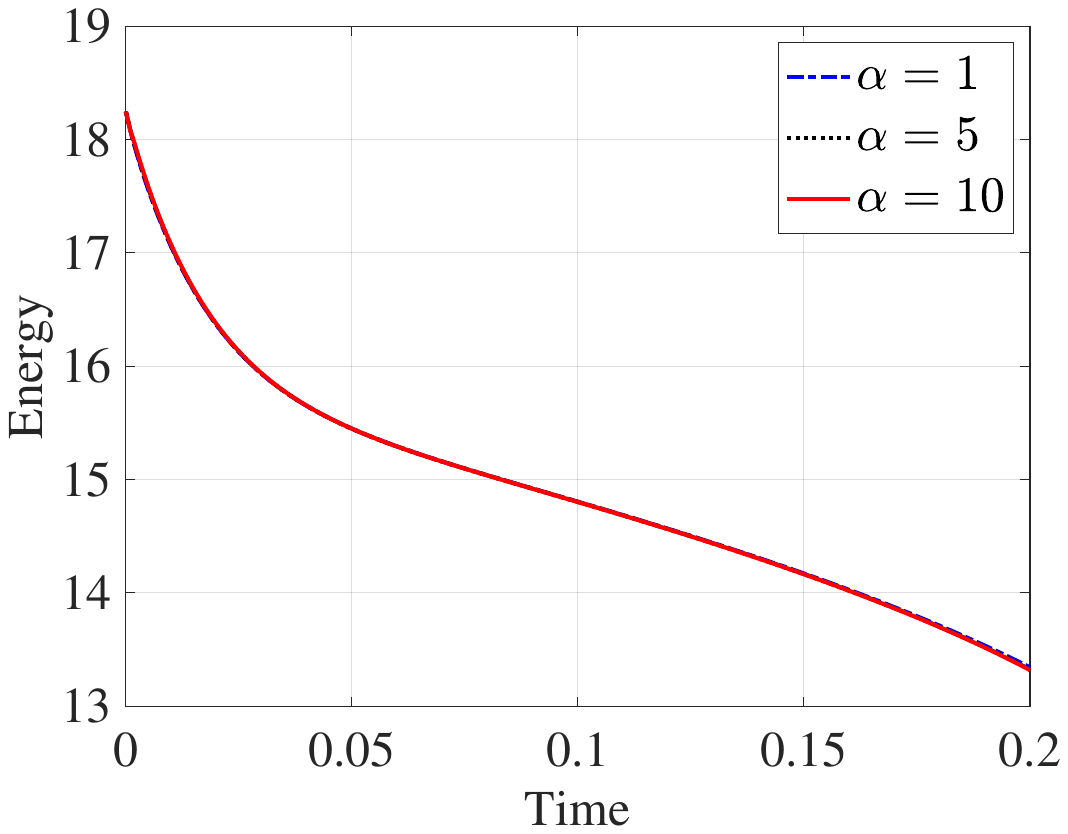}}
    \subfigure[Volume ]{\includegraphics[width=.32\textwidth,trim=0 200 0 200, clip]{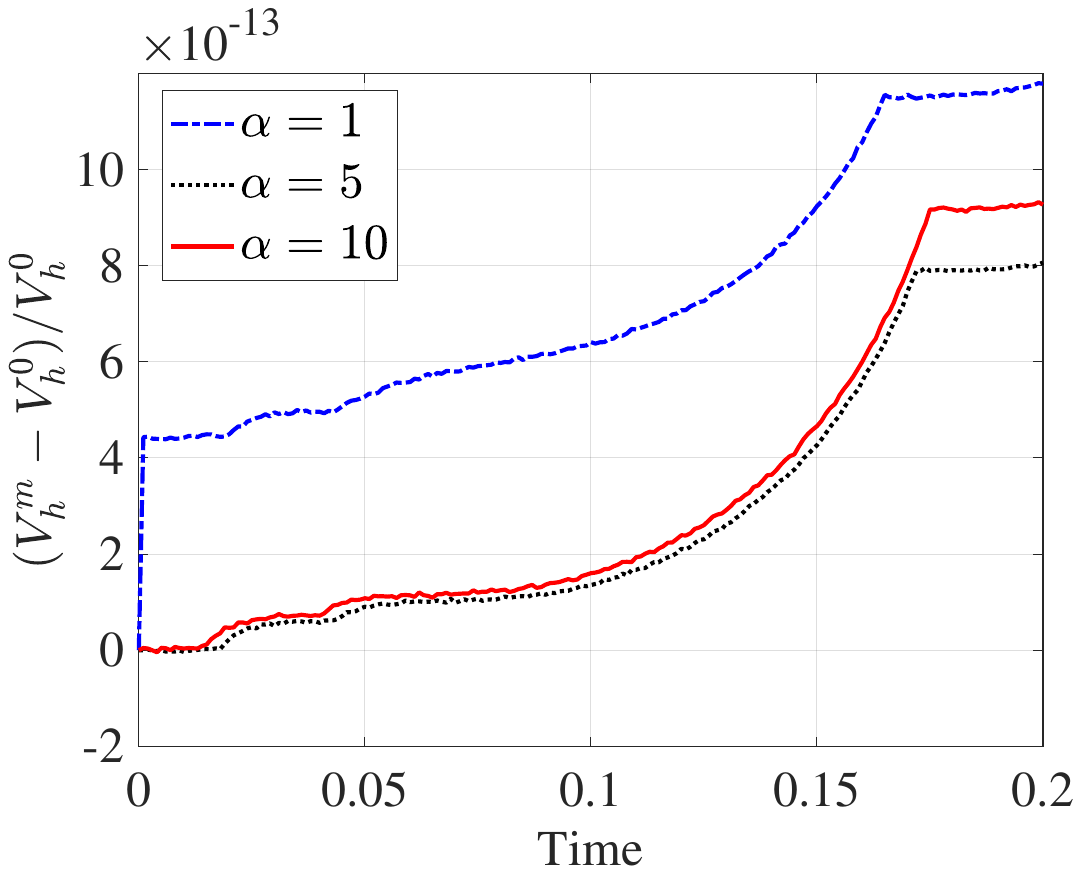}}
    \subfigure[Mesh quality]{\includegraphics[width=.32\textwidth,trim=0 200 0 200, clip]{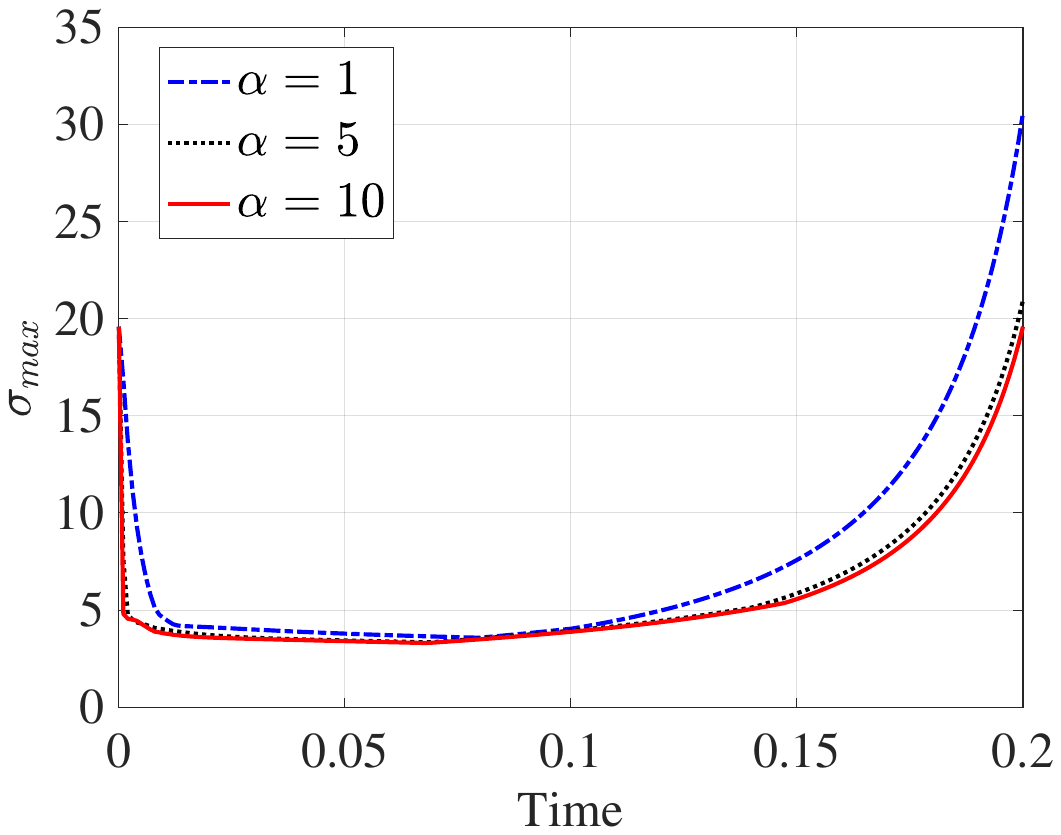}}
    \caption{Surface area with respect to time (a), relative volume loss (b)  and mesh quality (c) for \Cref{ex-Torus-MCF} with  $\tau=1\times 10^{-3}$ and $tol=5\times 10^{-12}$.}
    \label{fig-Torus-MCf-EVI}
\end{figure}
\end{example}
\begin{example}\label{ex-box-MCF}
    We consider the evolution of a $1:1:10$ box. The initial mesh is given in \Cref{fig-box-MCF} (a) with $(N_p,N_T)=(1346,2688)$. We take the initial mesh as $\cM_h$ and run our scheme by setting $\alpha=5$, $t_m=T\big(\frac{m}{M})^2,\,m=0,1,\cdots,M$ and $tol=5\times 10^{-12}$ with $T=5,\,M=1000$. \Cref{fig-box-MCF} (b)-(e) presents the numerical solutions at various time. Since the initial surface is uniformly convex, we can see no topological change occurs.   \Cref{fig-box-MCf-EVI} plots the energy, relative volume loss and $\sigma_{{\rm max}}$ under $\alpha=1,\,5,\, 10$. One can see that despite the large deformation of the surface,  different $\alpha$ does not have significant influence on the evolution in this example. The relative volume loss is on the same level with the iterative tolerance. The robustness of our scheme on mesh distribution is again verified. 
    \begin{figure}[!htp]
        \centering
        \subfigure[Initial surface]{\includegraphics[width=.31\textwidth,trim=0 250 0 250, clip]{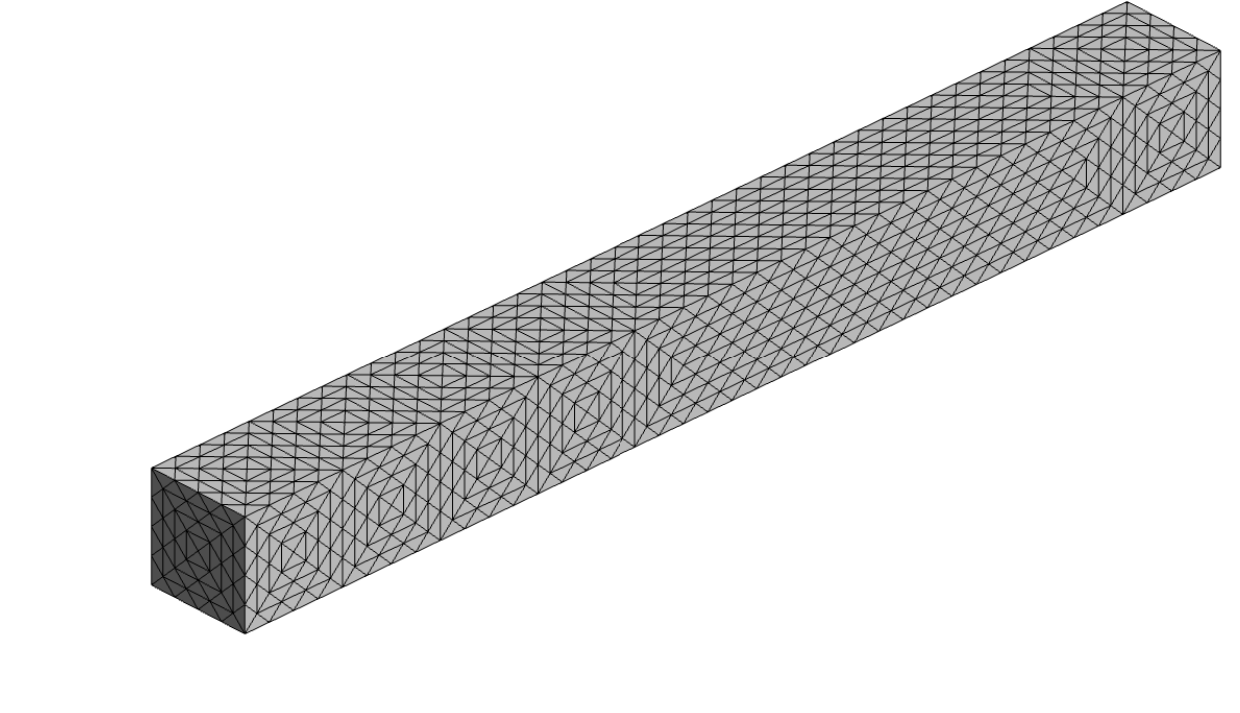}}\\
        \subfigure[$t=0.3125$]{\includegraphics[width=.26\textwidth,trim=20 90 20 80, clip]{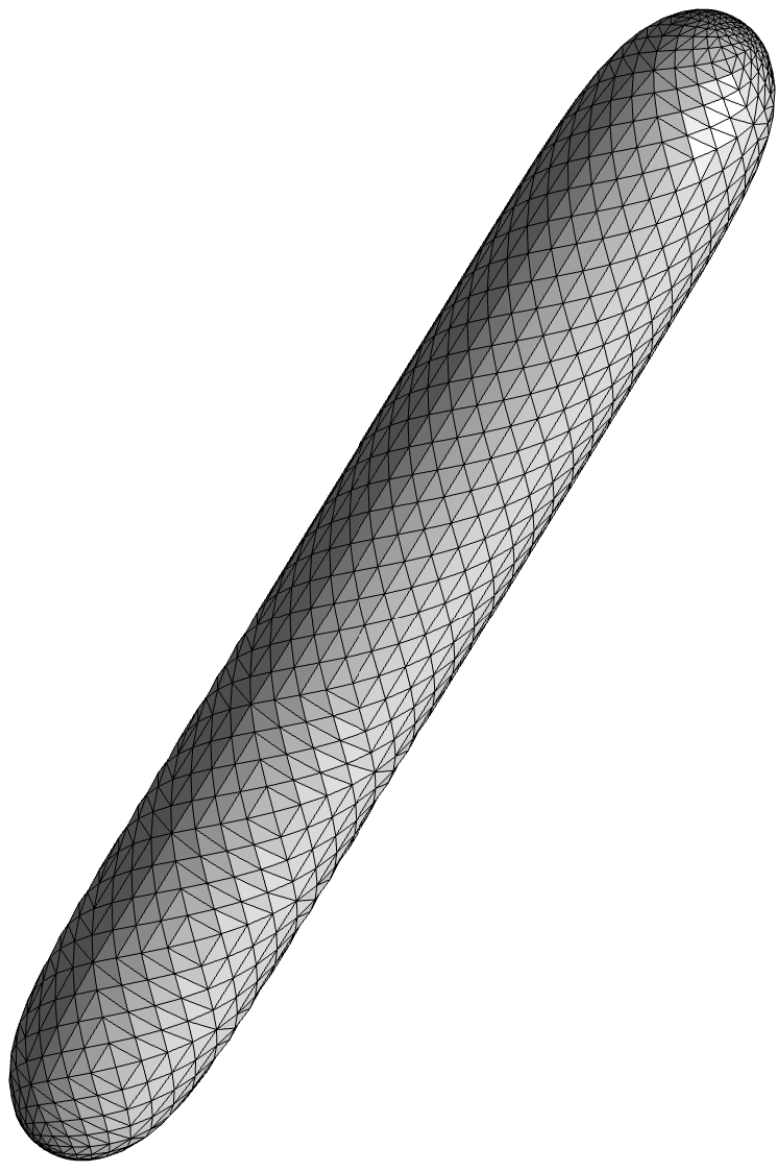}}
        \subfigure[$t=0.8$]{\includegraphics[width=.22\textwidth,trim=20 90 20 80, clip]{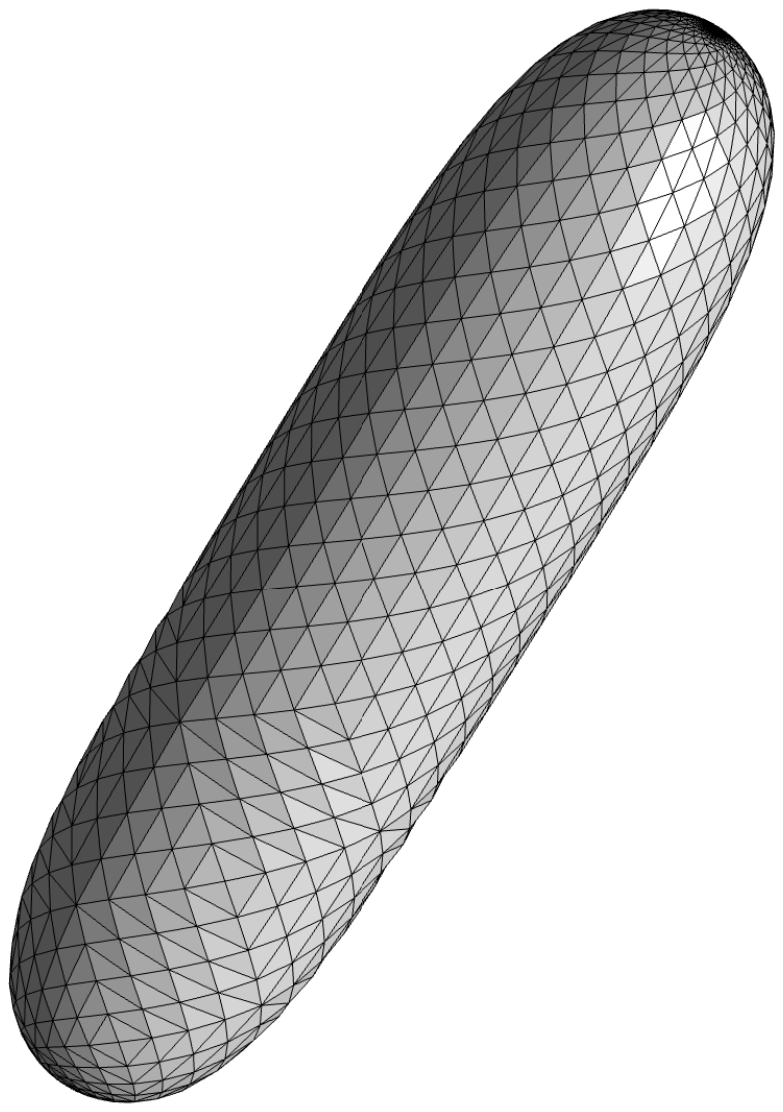}}
        \subfigure[$t=1.8$]{\includegraphics[width=.16\textwidth,trim=0 90 0 40, clip]{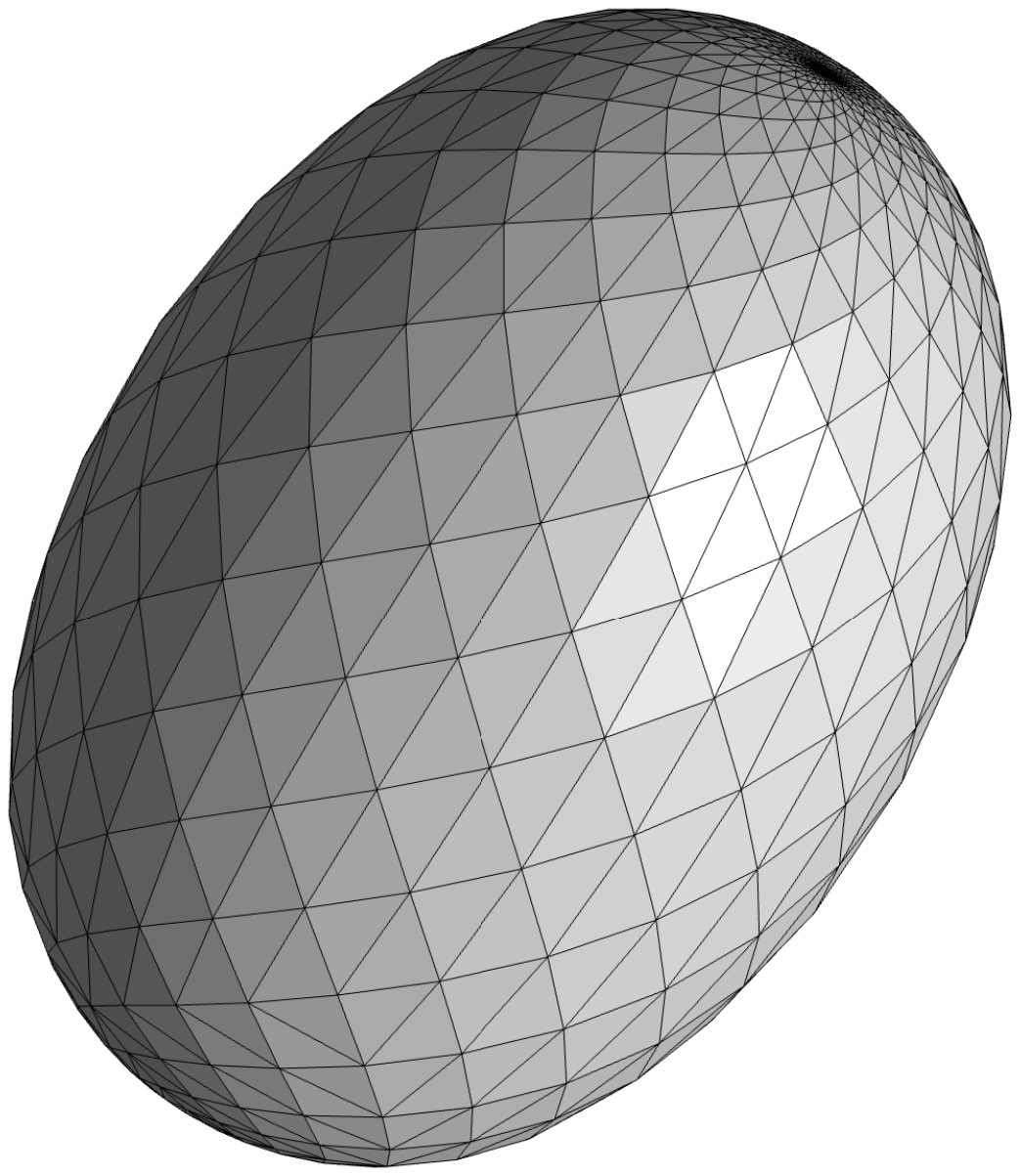}}
        \hspace{1.0cm}
        \subfigure[$t=5$]{\includegraphics[width=.16\textwidth,trim=0 90 0 40, clip]{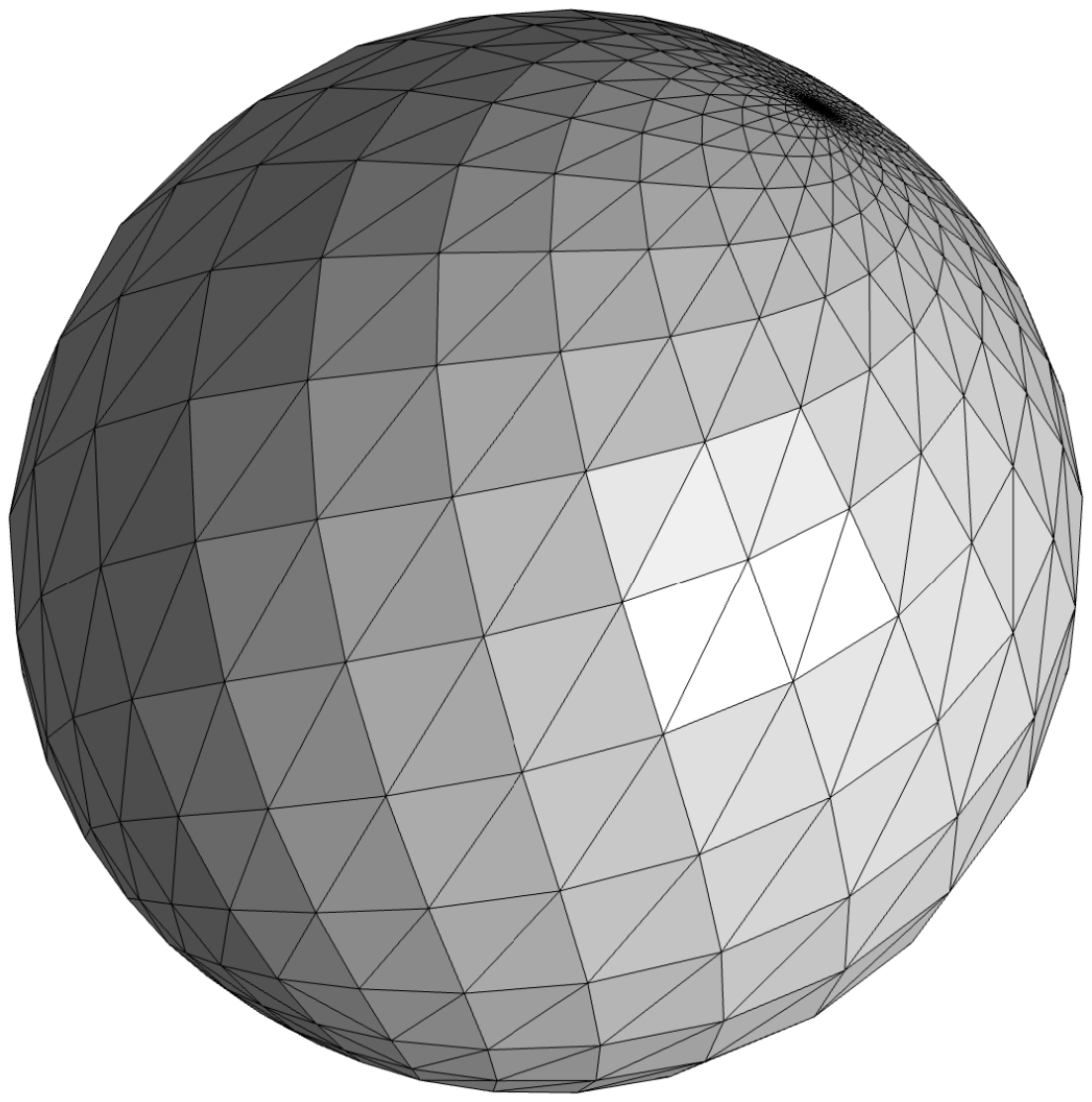}}
        \caption{Initial surface (a),  Plots (b)-(e) are numerical surfaces at $t=0.3125$, $t=0.8$, $t=1.8$ and $t=5$.}
        \label{fig-box-MCF}
    \end{figure}

\begin{figure}[!htp]
    \centering
    \subfigure[Surface area]{\includegraphics[width=.32\textwidth,trim=0 200 0 200, clip]{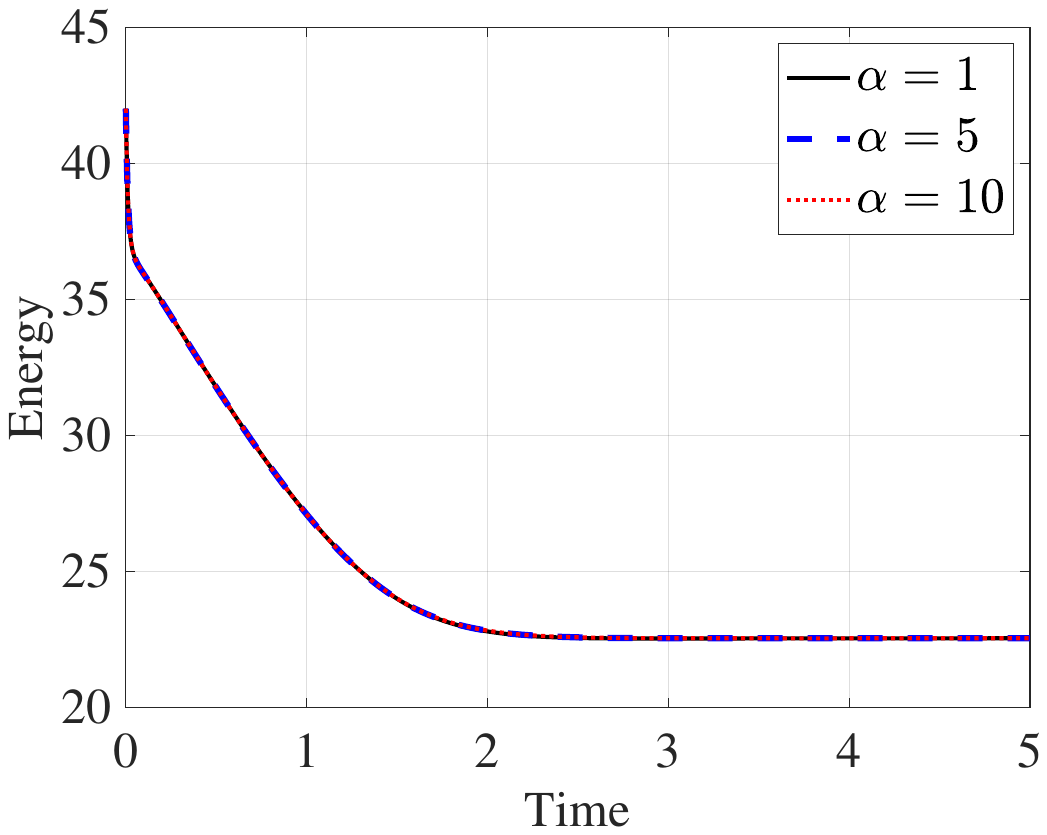}}
    \subfigure[Volume]{\includegraphics[width=.32\textwidth,trim=0 200 0 200, clip]{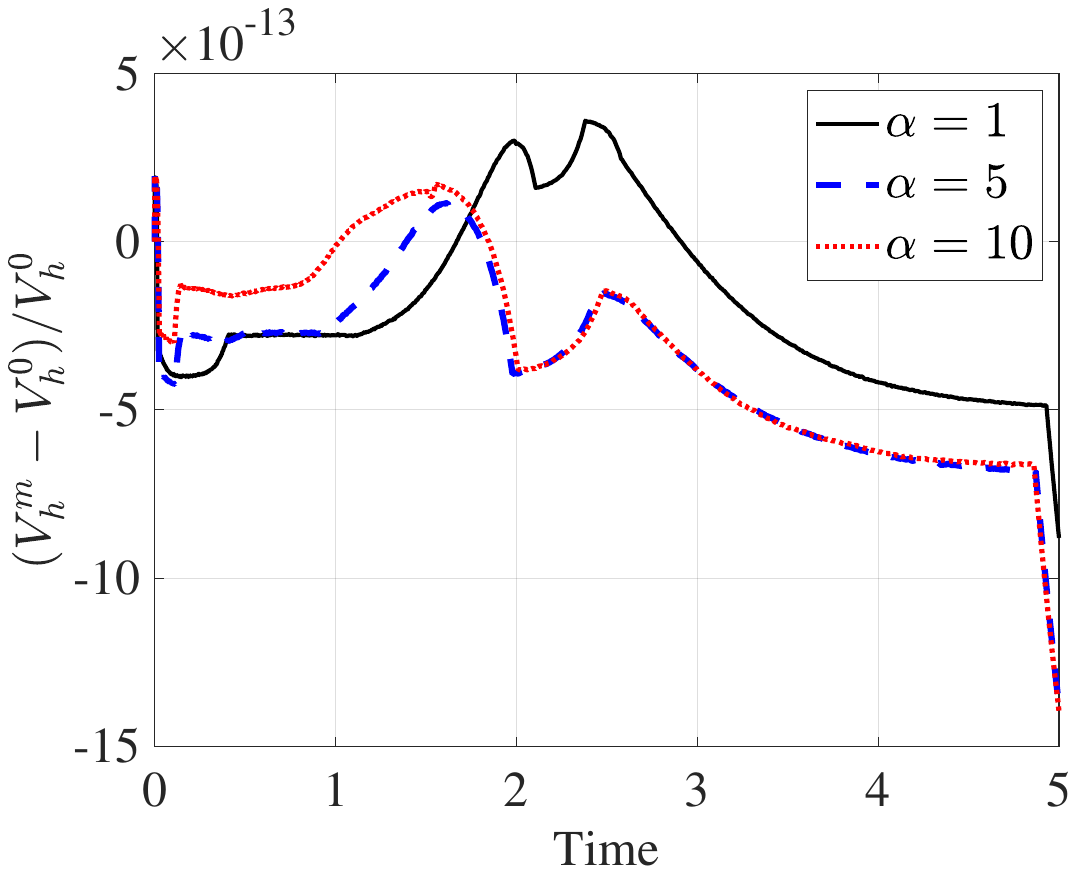}}
    \subfigure[Mesh quality]{\includegraphics[width=.32\textwidth,trim=0 200 0 200, clip]{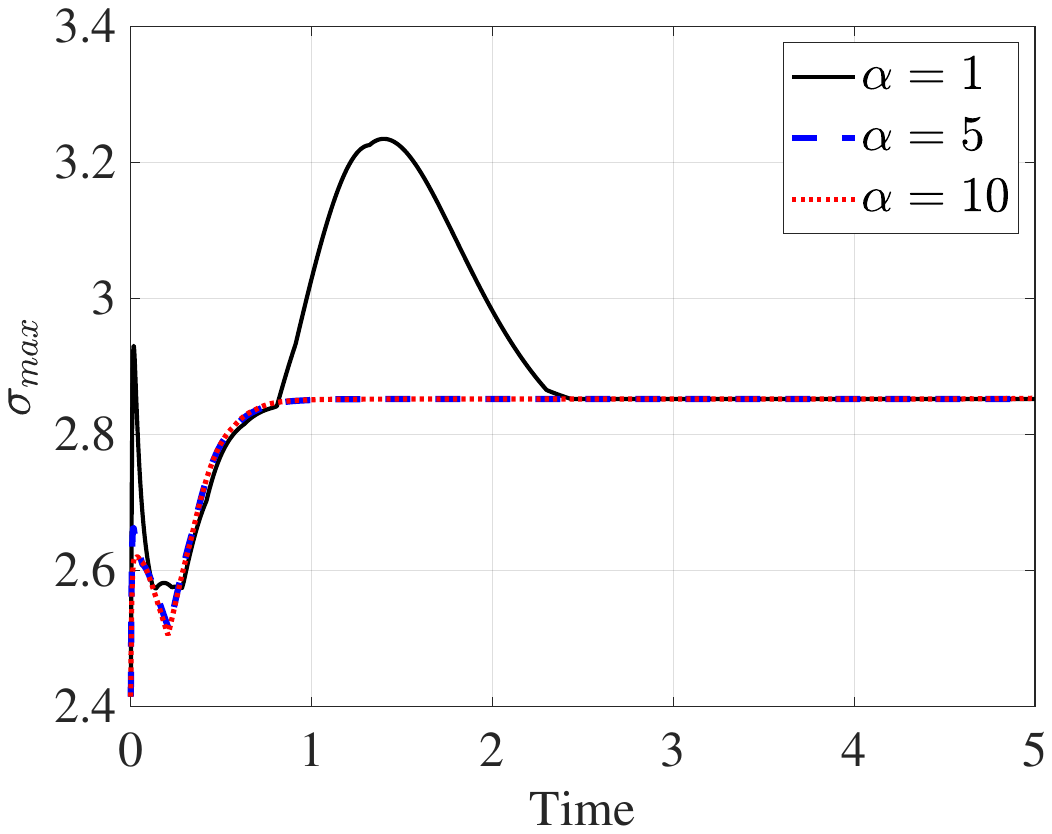}}
    \caption{Surface area with respect to time (a), relative volume loss (b)  and mesh quality (c) for \Cref{ex-box-MCF} with $tol=5\times 10^{-12}$.}
    \label{fig-box-MCf-EVI}
\end{figure}
\end{example}
\section{Conclusion}
In this paper, we have presented a structure-preserving Crank-Nicolson method for surface diffusion and the  volume-preserving mean curvature flow. The discretization is based on a reformulation of the original evolution equation, in which the tangential motion is governed by a harmonic map heat flow. Since the heat flow decreases the harmonic energy continuously, the mesh can always be maintained well. Volume conservation is ensured by a careful design of the scheme itself, and the energy decrease is enforced via a Lagrange multiplier. Numerical experiments verified the structure-preserving properties and also demonstrated the advantages of our scheme in terms of mesh distribution.

\bibliographystyle{siamplain}
\bibliography{References_GFs}    
\end{document}